\documentclass[a4paper,twoside,american,english]{scrartcl}
\usepackage[T1]{fontenc}
\usepackage[latin1]{inputenc}
\usepackage{babel}
\usepackage{verbatim}
\usepackage{mathtools}
\usepackage{amsthm}
\usepackage{amsmath}
\usepackage{amssymb}
\usepackage[unicode=true,pdfusetitle,
 bookmarks=true,bookmarksnumbered=false,bookmarksopen=false,
 breaklinks=false,pdfborder={0 0 1},backref=false,colorlinks=false]
 {hyperref}

\makeatletter

\pdfpageheight\paperheight
\pdfpagewidth\paperwidth

\numberwithin{equation}{section}
\numberwithin{figure}{section}
\theoremstyle{plain}
\newtheorem{thm}{\protect\theoremname}[section]
  \theoremstyle{definition}
  \newtheorem{defn}[thm]{\protect\definitionname}
  \theoremstyle{definition}
  \newtheorem{example}[thm]{\protect\examplename}
  \theoremstyle{plain}
  \newtheorem{lem}[thm]{\protect\lemmaname}
  \theoremstyle{plain}
  \newtheorem{cor}[thm]{\protect\corollaryname}
  \theoremstyle{plain}
  \newtheorem{prop}[thm]{\protect\propositionname}
  \theoremstyle{remark}
  \newtheorem{rem}[thm]{\protect\remarkname}

\@ifundefined{date}{}{\date{}}

\usepackage[T1]{fontenc}    
\usepackage{enumerate}
\usepackage{a4wide}        
\usepackage{amsmath}
\usepackage{euscript}
\usepackage{url}
\usepackage{scrpage2}
\usepackage{tu-preprint}
\usepackage{scrpage2}
\allowdisplaybreaks[1] 
\usepackage{amsfonts}

\newcommand{\hide}[1]{}

\newenvironment{keywords}{ \noindent\footnotesize\textbf{Keywords and phrases:}}{}

\newenvironment{class}{\noindent\footnotesize\textbf{Mathematics subject classification 2010:}}{}

\newcommand*{\trace}{\operatorname{trace}}
\newcommand*{\restricted}[1]{{\mid_{#1}}}

\newcommand*{\dive}{\operatorname{div}}
\renewcommand*{\div}{\operatorname{div}}
\newcommand*{\curl}{\operatorname{curl}}

\newcommand*{\grad}{\operatorname{grad}}

\DeclareMathAccent{\Circ}{\mathalpha}{operators}{"17}
\newcommand{\interior}[1]{\Circ{#1}}
\DeclareMathOperator{\id}{id}

\renewcommand{\Re}{\operatorname{\mathfrak{Re}}}

\newcommand{\oi}[2]{\left]#1,#2 \right[}

\renewcommand{\i}[0]{\mathrm{i}}

\renewcommand{\tilde}{\widetilde}
\renewcommand*{\epsilon}{\varepsilon}
\renewcommand*{\rho}{\varrho}

\makeatother

\institut{Institut f\"ur Analysis}

\preprintnumber{MATH-AN-02-2015}

\preprinttitle{On Abstract $\grad$--$\div$ Systems.}

\author{Rainer Picard, Stefan Seidler, Sascha Trostorff, Marcus Waurick}

\makeatother

  \addto\captionsamerican{\renewcommand{\corollaryname}{Corollary}}
  \addto\captionsamerican{\renewcommand{\definitionname}{Definition}}
  \addto\captionsamerican{\renewcommand{\examplename}{Example}}
  \addto\captionsamerican{\renewcommand{\lemmaname}{Lemma}}
  \addto\captionsamerican{\renewcommand{\propositionname}{Proposition}}
  \addto\captionsamerican{\renewcommand{\remarkname}{Remark}}
  \addto\captionsamerican{\renewcommand{\theoremname}{Theorem}}
  \addto\captionsenglish{\renewcommand{\corollaryname}{Corollary}}
  \addto\captionsenglish{\renewcommand{\definitionname}{Definition}}
  \addto\captionsenglish{\renewcommand{\examplename}{Example}}
  \addto\captionsenglish{\renewcommand{\lemmaname}{Lemma}}
  \addto\captionsenglish{\renewcommand{\propositionname}{Proposition}}
  \addto\captionsenglish{\renewcommand{\remarkname}{Remark}}
  \addto\captionsenglish{\renewcommand{\theoremname}{Theorem}}
  \providecommand{\corollaryname}{Corollary}
  \providecommand{\definitionname}{Definition}
  \providecommand{\examplename}{Example}
  \providecommand{\lemmaname}{Lemma}
  \providecommand{\propositionname}{Proposition}
  \providecommand{\remarkname}{Remark}
\providecommand{\theoremname}{Theorem}

\begin{document}
\selectlanguage{american}%
\makepreprinttitlepage

\selectlanguage{english}%

\title{On Abstract $\mathrm{grad}$--$\mathrm{div}$ Systems.}

\author{Rainer Picard, Stefan Seidler, Sascha Trostorff, Marcus Waurick}

\maketitle
\selectlanguage{american}%
\begin{abstract}\foreignlanguage{english}{ \textbf{Abstract. }For
a large class of dynamical problems from mathematical physics the
skew-selfadjointness of a spatial operator of the form $A=\left(\begin{array}{cc}
0 & -C^{*}\\
C & 0
\end{array}\right)$, where $C:D\left(C\right)\subseteq H_{0}\to H_{1}$ is a closed densely
defined linear operator, is a typical property. Guided by the standard
example, where $C=\grad=\left(\begin{array}{c}
\partial_{1}\\
\vdots\\
\partial_{n}
\end{array}\right)$ (and $-C^{*}=\dive$, subject to suitable boundary constraints),
an abstract class of operators $C=\left(\begin{array}{c}
C_{1}\\
\vdots\\
C_{n}
\end{array}\right)$ is introduced (hence the title). As a particular application we consider
a non-standard coupling mechanism and the incorporation of diffusive
boundary conditions both modeled by setting associated with a skew-selfadjoint
spatial operator $A$.}\end{abstract}


\begin{class}\foreignlanguage{english}{ 35F45, 46N20, 47N20} \end{class}

\selectlanguage{english}%
\newpage

\tableofcontents{}

\newpage

\setcounter{section}{-1}

\section{Introduction}

In a number of studies \cite{Waurick2014MMAS_Frac,Waurick2014MMAS_MaxGrav,Waurick2014MMS_Thermo-Ela,Waurick2014SIAM_HomFrac,Trostorff,Picard2013}
it has been illustrated, that typical initial boundary value problems
of mathematical physics can be represented in the general form
\begin{equation}
\left(\partial_{0}\mathcal{M}+A\right)U=F,\label{eq:evo}
\end{equation}
where $A$ is skew-selfadjoint, indeed commonly of the specific block
matrix form 
\begin{equation}
A=\left(\begin{array}{cc}
0 & -C^{*}\\
C & 0
\end{array}\right)\label{eq:skew-canon}
\end{equation}
with $C:X_{1}\subseteq X_{0}\to Y$ a closed densely defined linear
operator between Hilbert spaces $X_{0}$ and $Y$ with $X_{1}=D(C)$,
see e.g. \cite{Pi2009-1,PDE_DeGruyter}. The operator $\mathcal{M}$
is referred to as the material law operator, which in the situation
discussed here is a linear operator acting on a Hilbert space realizing
the space-time the problems are formulated in, \cite{Waurick2013JEE_Non-auto,Waurick2014MMAS_Non-auto,Waurick2014MMAS_Frac,Pi2009-1}.

The main purpose of this paper is to focus on the operator $C$ in
this construction of the operator $A:D\left(C\right)\times D\left(C^{*}\right)\subseteq X_{0}\oplus Y\to X_{0}\oplus Y$
when $Y$ is itself a direct sum of Hilbert spaces. In such a situation
we shall loosely refer to a system of the form (\ref{eq:evo}) as
an abstract $\grad$--$\dive$ system. The guiding example, which
at the same time motivates the name, is to take for $C:X_{1}\subseteq X_{0}\to Y$
the differential operator $\nabla=\left(\begin{array}{c}
\partial_{1}\\
\partial_{2}\\
\partial_{3}
\end{array}\right)$ with domain $X_{1}=\interior H_{1}\left(\Omega\right)$, the Sobolev
space of weakly differentiable functions with $L^{2}$-derivatives
and vanishing boundary data. The spaces $X_{0}$ and $Y$ would be
in this case $L^{2}\left(\Omega\right)$ and $L^{2}\left(\Omega\right)^{3}\left(=L^{2}\left(\Omega\right)\oplus L^{2}\left(\Omega\right)\oplus L^{2}\left(\Omega\right)\right)$,
respectively. The corresponding equation of the form (\ref{eq:evo})
would lead in particular to a model for the acoustic wave propagation
or -- depending on the material law operator $\mathcal{M}$ -- for
the dissipation of heat with so-called Dirichlet boundary condition,
see e.g.~\cite{Pi2009-1,Waurick2014PROC_Survey,Waurick2013_ContDep}.
The adjoint operator is the negative weak divergence $C^{*}=-\dive$.
It is 
\begin{equation}
C=\left(\begin{array}{c}
\partial_{1}\restricted{X_{1}}\\
\partial_{2}\restricted{X_{1}}\\
\partial_{3}\restricted{X_{1}}
\end{array}\right).\label{eq:grad}
\end{equation}
The idea in this paper is to replace the role of the partial derivatives
in (\ref{eq:grad}) by general operators in general Hilbert spaces,
hence the term abstract grad-div systems for the corresponding evolutionary
systems associated with the skew-selfadjoint operator A constructed
according to (\ref{eq:skew-canon}). That the study of this particular
class of skew-selfadjoint operators leads to interesting applications
is illustrated by three examples. First we consider the so-called
Guyer-Krumhansl model of thermo-dynamics, \cite{PhysRev.133.A1411,PhysRev.148.766,PhysRev.148.778},
as a particular instance of this construction. The last two examples
illustrate that the concept of $\grad$--$\dive$ systems is well-suited
to study dynamic boundary conditions. A first implementation, leading
up to the concept of $\grad$--$\dive$ systems introduced here, has
been established in the context of a particular class of boundary
control problems, \cite{zbMATH06295127,zbMATH06203877}. Our investigation
of $\grad$--$\dive$ systems has been motivated by the discussion
of a heat conduction problem with a dissipative dynamic boundary condition
in \cite{2013arXiv1312.5882D}. The generalization to the context
of evolutionary equations is the topic of our second example. The
last application example is connected to the Leontovich boundary condition
of electrodynamics, see e.g.~\cite{zbMATH03105467,Gilliam1982129,zbMATH00768021}.
We shall discuss two different implementations. The first one is based
on classical boundary trace concepts (see e.g.~\cite{zbMATH01866780,WeissStaffans2013}),
the second approach uses a boundary data concept developed in \cite{zbMATH06295127,zbMATH06203877}
with extensive use in \cite{Trostorff2014}, which has the advantage
that no constraints on the quality of the boundary are incurred. 

In Section \ref{sec:Construction-of-Abstract} we start by presenting
the fundamental construction of abstract $\mathrm{grad}$--$\mathrm{div}$
systems. The remaining section is then dedicated to illustrating the
usefulness of the concept by examples.

\section{\label{sec:Construction-of-Abstract}Construction of Abstract $\mathrm{grad}$--$\mathrm{div}$
Systems.}

In this section, we shall reconsider the concept of the adjoint operator
of a densely defined, closed linear operator $C$, specifically in
order to deal with the fact that the image space $Y$ of the operator
$C$ is given as an orthogonal sum of Hilbert spaces. Let us first
provide a precise definition of what we would like to call an abstract
$\grad$-$\dive$-system.
\begin{defn}
Let $C:X_{1}\subseteq X_{0}\to Y$, be a densely defined, closed linear
operator with domain $X_{1}$ between Hilbert spaces $X_{0},Y$. We
shall refer to a system of the form (\ref{eq:evo}) with $A$ generated
via (\ref{eq:skew-canon}), as an \emph{abstract $\grad$--$\dive$
system, }if $Y$ given as a direct sum, i.e. $Y\coloneqq\bigoplus_{k\in\{1,\ldots,n\}}Y_{k}$,
for Hilbert spaces $Y_{k}$, $k\in\{1,\ldots,n\}$, $n\in\mathbb{N}$\emph{. }
\end{defn}
As a matter of jargon we shall say that the abstract\emph{ $\grad$--$\dive$
}system\emph{ }is \emph{generated by }$C$. If $\iota_{Y_{k}}$ denotes
the canonical isometric embedding of \emph{$Y_{k}$ }into\emph{ $Y$}
then, with $C_{k}\coloneqq\iota_{Y_{k}}^{*}C$, $k\in\{1,\ldots,n\}$,
we have 
\[
Cx=C_{1}x\oplus\cdots\oplus C_{n}x=\left(\begin{array}{c}
C_{1}x\\
\vdots\\
C_{n}x
\end{array}\right)=\left(\begin{array}{c}
C_{1}\\
\vdots\\
C_{n}
\end{array}\right)x\in\left(\begin{array}{c}
Y_{1}\\
\vdots\\
Y_{n}
\end{array}\right)=Y
\]
for $x\in X_{1}$. 

The aim of this section is to give a characterization of the adjoint
$C^{*}$ of $C$. However, before stating and proving the respective
theorem, we shall explore as a first elementary example the name-giving
case mentioned in the introduction:
\begin{example}
\label{exa:gradient}Let $n\in\mathbb{N}$, $\Omega\subseteq\mathbb{R}^{n}$
non-empty and open, $X_{0}\coloneqq L^{2}\left(\Omega\right)$ and
$X_{1}\coloneqq\interior H_{1}(\Omega)$, the Sobolev space of weakly
differentiable functions in $L^{2}(\Omega)$ with vanishing Dirichlet
trace, i.e., the closure of compactly supported continuously differentiable
functions with respect to the $H_{1}$-norm $f\mapsto\sqrt{\sum_{\alpha\in\mathbb{N}_{0}^{n},|\alpha|\leq1}|\partial^{\alpha}f|_{L^{2}(\Omega)}^{2}}$.
Then clearly, with $Y=\bigoplus_{k\in\{1,\ldots,n\}}L^{2}\left(\Omega\right)$
the operator $C\coloneqq\interior\grad:\interior H_{1}(\Omega)\subseteq L^{2}\left(\Omega\right)\to\bigoplus_{k\in\{1,\ldots,n\}}L^{2}\left(\Omega\right)$,
defined as the closure of the classical vector-analytical operation
$\grad$ restricted to smooth functions with compact support in $\Omega$,
generates an abstract $\grad$-$\dive$-system. In the same way, choosing
$X_{1}=H_{1}(\Omega)$ instead, we get that $C=\grad$ also generates
an abstract $\grad$-$\dive$-system. We recall that the difference
between the operators $\grad$ and $\interior\grad$ lies in the additionally
prescribed homogeneous Dirichlet boundary condition for elements in
the domain of $\interior\grad.$
\end{example}
To clarify notation, we need the following definition.
\begin{defn}
Let $X,Y$ be Hilbert spaces and $L:X\to Y$ a bounded linear operator.
Then we denote the dual operator of $L$ by $L':Y'\to X'$ defined
by 
\[
\left(L'y'\right)(x)\coloneqq y'(Lx)\quad(y'\in Y',x\in X).
\]
Moreover, we may consider a densely defined linear operator $S:X_{1}\subseteq X_{0}\to Y$
such that 
\begin{align*}
L_{S}:X_{1} & \to Y\\
x & \mapsto Sx
\end{align*}
is a continuous linear operator ($S$ need not be closable) and define
the operator $S^{\diamond}:Y\to X_{1}'$ by $S^{\diamond}\coloneqq L_{S}^{\prime}\circ R_{Y'}.$
Here $R_{Y^{\prime}}:Y\to Y^{\prime}$ denotes the Riesz isomorphism,
given by 
\[
\left(R_{Y^{\prime}}y\right)(z)\coloneqq\langle y|z\rangle_{Y}\quad(y,z\in Y).
\]

\end{defn}
Note that we have by definition 
\[
\left(S^{\diamond}y\right)(x)=\left(L_{S}^{\prime}(R_{Y^{\prime}}y)\right)(x)=\left(R_{Y^{\prime}}y\right)(Sx)=\langle y|Sx\rangle_{Y}
\]
for $y\in Y,x\in X.$ As a matter of convenience this can be written
suggestively as
\[
\left\langle S^{\diamond}y|x\right\rangle _{X_{0}}=\langle y|Sx\rangle_{Y},
\]
where $\left\langle \:\cdot\:|\:\cdot\:\right\rangle _{X_{0}}$ denotes
here the continuous extension of the inner product of $X_{0}$ to
a duality pairing on $X_{1}^{\prime}\times X_{1}$. Note that $X_{1}\hookrightarrow X_{0}\stackrel{R_{X'_{0}}}{=}X'_{0}\hookrightarrow X_{1}^{\prime}$
is a Gelfand triple. 
\begin{lem}
\label{lem:adjoint-dual}Let $X_{0},X_{1},Y$ be Hilbert spaces, $X_{1}\subseteq X_{0}$
dense, $C\colon X_{1}\subseteq X_{0}\to Y$ a closed linear operator.
Then the adjoint $C^{*}$ of the operator $C$ is densely defined
and its adjoint $C^{*}$ is given by
\[
C^{*}=C^{\diamond}\cap\left(Y\oplus X_{0}\right),
\]
which is the same as to say
\[
C^{*}=\left\{ (y,x)\in Y\oplus X_{1}'\,\left|\, C^{\diamond}y=x\text{ and }x\in X_{0}\right.\right\} .
\]
\end{lem}
\begin{proof}
Note that, by definition, $C^{*}\subseteq C^{\diamond}$ and $C^{*}\subseteq Y\oplus X_{0}$.
Hence, $C^{*}\subseteq C^{\diamond}\cap(Y\oplus X_{0})$. For the
remaining implication, let $(y,x)\in Y\oplus X_{0}$. Then we have
\begin{align*}
(y,x)\in C^{*} & \iff\forall\phi\in X_{1}\colon\langle y|C\phi\rangle_{Y}=\langle x|\phi\rangle_{X_{0}}\\
 & \iff\forall\phi\in X_{1}\colon C^{\diamond}y(\phi)=\langle x|\phi\rangle_{X_{0}}.
\end{align*}
This establishes the assertion.
\end{proof}
The latter observation has been employed in more concrete situation
in control theory, see e.g.~\cite{zbMATH06203877,zbMATH06295127,Waurick2014IMAJMCI_ComprehContr}
and, in particular, the references in \cite{Waurick2014IMAJMCI_ComprehContr}.
\begin{lem}
\label{lem:Orth-sum}Let $C:X_{1}\subseteq X_{0}\to Y$ generate an
abstract $\grad$--$\dive$ system with $Y=\bigoplus_{k\in\left\{ 1,\ldots,n\right\} }Y_{k}$.
Then
\[
C^{\diamond}=\left(\begin{array}{ccc}
C_{1}^{\diamond} & \cdots & C_{n}^{\diamond}\end{array}\right)\colon Y\ni\left(\begin{array}{c}
y_{1}\\
\vdots\\
y_{n}
\end{array}\right)\mapsto\sum_{k=1}^{n}C_{k}^{\diamond}y_{k}\in X'_{1}
\]
\end{lem}
\begin{proof}
We have for $x\in X_{1}$, $y=\left(\begin{array}{c}
y_{1}\\
\vdots\\
y_{n}
\end{array}\right)=y_{1}\oplus\cdots\oplus y_{n}\in Y$,
\[
(C^{\diamond}y)(x)=\left\langle y|Cx\right\rangle _{Y}=\sum_{k=1}^{n}\left\langle y_{k}|C_{k}x\right\rangle _{Y_{k}}=\sum_{k=1}^{n}\left(C_{k}^{\diamond}y_{k}\right)\left(x\right),
\]
hence, $C^{\diamond}y=\sum_{k=1}^{n}\, C_{k}^{\diamond}y_{k}$. 
\end{proof}
The latter two lemmas yield the proof of the following theorem:
\begin{thm}
\label{thm:AGDS-adjoint}Let $C$ generate an abstract $\grad$--$\dive$
system with $C=\left(\begin{array}{c}
C_{1}\\
\vdots\\
C_{n}
\end{array}\right)$. Then 
\begin{align*}
C^{*} & =\left(\begin{array}{ccc}
C_{1}^{\diamond} & \cdots & C_{n}^{\diamond}\end{array}\right)\cap\left(Y\oplus X_{0}\right)\\
 & =\left\{ \left(\left(y_{1},\ldots,y_{n}\right),x\right)\in Y\oplus X_{0}\,\left|\, x=\sum_{k=1}^{n}C_{k}^{\diamond}y_{k}\in X_{0}\right.\right\} .
\end{align*}

\end{thm}
Let us apply the latter theorem to the case $C=\interior\grad$, as
it was defined in Example \ref{exa:gradient}.
\begin{example}
Let $X_{0}=L^{2}(\Omega),X_{1}\coloneqq\interior H_{1}(\Omega)$ and
$Y_{k}\coloneqq L^{2}(\Omega)$ for $k\in\{1,\ldots,n\}$ and some
non-empty, open set $\Omega\subseteq\mathbb{R}^{n}.$ Recall that
$C=\interior\grad$ generates an abstract $\grad$-$\dive$-system.
As in this case, $C=\left(\begin{array}{c}
C_{1}\\
\vdots\\
C_{n}
\end{array}\right)$ with $C_{k}:\interior H_{1}(\Omega)\subseteq L^{2}\left(\Omega\right)\to L^{2}(\Omega)$
is defined as $C_{k}f=\partial_{k}f$ for $f\in\interior H_{1}(\Omega),$
we obtain in particular that 
\[
\left(C_{k}^{\diamond}g\right)(f)=\langle g|\partial_{k}f\rangle_{L^{2}(\Omega)}=-\left(\partial_{k}g\right)(f)
\]
for $g\in L^{2}(\Omega),f\in\interior C_{\infty}(\Omega),$ where
here $\partial_{k}g$ is meant in the sense of distributions. Here
$\interior C_{\infty}(\Omega)$ denotes the set of arbitrarily differentiable
functions with compact support in $\Omega$. In consequence 
\[
C^{\diamond}\left(\begin{array}{c}
g_{1}\\
\vdots\\
g_{n}
\end{array}\right)=-\sum_{k=1}^{n}\partial_{k}g_{k}
\]
and thus, $C^{\ast}=-\dive$ is the negative distributional divergence
on $L^{2}(\Omega)$ vector-fields restricted to those, whose divergence
is representable as an $L^{2}(\Omega)$-function. 

An immediate application of Theorem \ref{thm:AGDS-adjoint} is the
following corollary, which will be of interest in the next section:\end{example}
\begin{cor}
\label{cor:adjoint is res}Let $C$ generate an abstract $\grad$-$\dive$-system.
Assume that there exists a closed, densely defined operator $\interior C_{1}$
with
\[
\left(\begin{array}{c}
\interior C_{1}\\
0\\
\vdots\\
0
\end{array}\right)\subseteq\left(\begin{array}{c}
C_{1}\\
\vdots\\
C_{n}
\end{array}\right)=C.
\]
Then 
\[
C^{*}=C^{\diamond}\cap\left(\interior C_{1}^{*}\:0\cdots\:0\right)\subseteq\left(\interior C_{1}^{*}\:0\cdots\:0\right).
\]
\end{cor}
\begin{proof}
It suffices to observe that $C^{\diamond}\subseteq\left(\interior C_{1}^{\diamond}\:0\cdots\:0\right)$
and to apply Theorem \ref{thm:AGDS-adjoint}.
\end{proof}
The next statement contains a typical situation of operators of the
form $C=\left(\begin{array}{c}
C_{1}\\
\vdots\\
C_{n}
\end{array}\right)$ giving rise to the generation of abstract $\grad$--$\dive$ systems.
\begin{prop}
\label{prop:standard_grad-div}Let $X_{0},Y_{0},\ldots,Y_{n}$ be
Hilbert spaces, $C_{0}\colon D(C_{0})\subseteq X_{0}\to Y_{0}$ densely
defined and closed. Denote $X_{1}\coloneqq\left(D(C_{0}),\sqrt{|\cdot|^{2}+|C_{0}\cdot|^{2}}\right)$
and let $C_{k}\in L(X_{1},Y_{k}),$ $k\in\{1,\ldots,n\}$. Then
\[
C=\left(\begin{array}{c}
C_{0}\\
\vdots\\
C_{n}
\end{array}\right)\colon D(C_{0})\subseteq X_{0}\to\bigoplus_{k\in\{0,\ldots,n\}}Y_{k}
\]
generates an abstract $\grad$--$\dive$ system.\end{prop}
\begin{proof}
As $C_{0}$ is densely defined, so is $C$. Thus, closedness of $C$
is the only thing remaining to check. For this let $(x_{j})_{j\in\mathbb{N}}$
be a sequence in $D(C_{0})$ convergent in $X_{0}$ its limit being
denoted by $x\in X_{0}$ and with the property that $y_{k}\coloneqq\lim_{j\to\infty}C_{k}x_{j}$
exists in $Y_{k}$, $j\in\{0,\ldots,n\}$. By the closedness of $C_{0}$,
we infer $x\in D(C_{0})$ and $C_{0}x=y_{0}$. Thus, $x_{j}\to x$
in $X_{1}$ as $j\to\infty$. The continuity of $C_{k}$ for all $k\in\{1,\ldots,n\}$
yields the assertion. \end{proof}
\begin{rem}
Note that the operators $C_{k}$, $k\in\{1,\ldots,n\}$, need not
be closable with domain $D(C_{0})$. The following example illustrates
this fact. Take $X_{0}\coloneqq L^{2}(0,1)$, $C_{0}\coloneqq\partial\colon H^{1}(0,1)\subseteq L^{2}(0,1)\to L^{2}(0,1),f\mapsto f'$
with $f'$ denoting the distributional derivative. Then $C_{0}$ is
densely defined and closed. It is known that $H_{1}(0,1)\subseteq C[0,1]$
(Sobolev embedding theorem). Hence, $C_{1}f\coloneqq\delta_{\{1/2\}}f\coloneqq f(1/2)$
for $f\in H_{1}(0,1)$ defines a continuous linear operator from $H_{1}(0,1)$
to $\mathbb{C}$. It is easy to see that $C_{1}\colon H_{1}(0,1)\subseteq L^{2}(0,1)\to L^{2}(0,1)$
is \emph{not }closable. 
\end{rem}

\section{Some Applications}

Before we come to specific applications, we introduce some classical
differential operators, which we will need in their description. Throughout,
let $\Omega$ be a non-empty open subset of $\mathbb{R}^{n}$, $n\in\mathbb{N}$.
\begin{defn}
We define the operator $\interior\grad$ as the closure of the operator
\begin{align*}
\interior C_{\infty}(\Omega)\subseteq L^{2}(\Omega) & \to L^{2}(\Omega)^{n}\\
\phi & \mapsto(\partial_{i}\phi)_{i\in\{1,\ldots,n\}}.
\end{align*}
Obviously, $D(\interior\grad)=\interior H_{1}(\Omega)$ and thus,
this definition coincides with the previous definition given in Example
\ref{exa:gradient}. Similarly, we define $\interior\dive$ as the
closure of 
\begin{align*}
\interior C_{\infty}(\Omega)^{n}\subseteq L^{2}(\Omega)^{n} & \to L^{2}(\Omega)\\
(\phi_{i})_{i\in\{1,\ldots,n\}} & \mapsto\sum_{i=1}^{n}\partial_{i}\phi_{i}.
\end{align*}
Moreover, we define the gradient on vector-fields, denoted by the
same symbol $\interior\grad$, as the closure of 
\begin{align*}
\interior C_{\infty}(\Omega)^{n}\subseteq L^{2}(\Omega)^{n} & \to L^{2}(\Omega)^{n\times n}\\
(\phi_{i})_{i\in\{1,\ldots,n\}} & \mapsto(\partial_{k}\phi_{i})_{i,k\in\{1,\ldots,n\}}
\end{align*}
and likewise the divergence of matrix-valued functions, again denoted
by $\interior\dive,$ as the closure of 
\begin{align*}
\interior C_{\infty}(\Omega)^{n\times n}\subseteq L^{2}(\Omega)^{n\times n} & \to L^{2}(\Omega)^{n}\\
\left(\phi_{ik}\right)_{i,k\in\{1,\ldots,n\}} & \mapsto\left(\sum_{k=1}^{n}\partial_{k}\phi_{ik}\right)_{i\in\{1,\ldots,n\}}.
\end{align*}
In the special case $n=3$ we define the operator $\interior\curl$
as the closure of 
\begin{align*}
\interior C_{\infty}(\Omega)^{3}\subseteq L^{2}(\Omega)^{3} & \to L^{2}(\Omega)^{3}\\
\left(\begin{array}{c}
\phi_{1}\\
\phi_{2}\\
\phi_{3}
\end{array}\right) & \mapsto\left(\begin{array}{c}
\partial_{2}\phi_{3}-\partial_{3}\phi_{2}\\
\partial_{3}\phi_{1}-\partial_{1}\phi_{3}\\
\partial_{1}\phi_{2}-\partial_{2}\phi_{1}
\end{array}\right).
\end{align*}

By integration by parts we derive the following relations 
\begin{align*}
\interior\grad & \subseteq-(\interior\dive)^{\ast}\eqqcolon\grad,\\
\interior\dive & \subseteq-(\interior\grad)^{\ast}\eqqcolon\dive,\\
\interior\curl & \subseteq(\interior\curl)^{\ast}\eqqcolon\curl.
\end{align*}
It is easy to see that for the scalar-valued gradient, we have $D(\grad)=H_{1}(\Omega)$
and thus, the definition of $\grad$ coincides with the previous one
in Example \ref{exa:gradient}. We recall that $u\in D(\interior\grad)$
satisfies $u|_{\partial\Omega}=0,$ $v\in D(\interior\dive)$ satisfies
$v|_{\partial\Omega}\cdot n=0$ and $w\in D(\interior\curl)$ satisfies
$w|_{\partial\Omega}\times n=0$ for domains $\Omega$ with smooth
boundary, where $n$ denotes the unit outward normal vector field
on $\partial\Omega.$ As the operators $\interior\grad,\,\interior\dive$
and $\interior\curl$ can be defined for arbitrary open sets $\Omega$,
we will use them to formulate the respective generalized boundary
conditions, which do not need to make use of classical boundary traces.
\end{defn}
Beyond the aforementioned spatial differential operators, we need
a suitable realization of the temporal derivative, see also \cite[Example 2.3]{PDE_DeGruyter}
or \cite[Section 2]{Kalauch2011} for a brief discussion.
\begin{defn}
Let $\rho>0$ and $H$ be a Hilbert space. We consider the weighted
$L^{2}$-space 
\[
H_{\rho,0}(\mathbb{R};H)\coloneqq\left\{ f:\mathbb{R}\to H\,|\, f\mbox{ measurable, }\int_{\mathbb{R}}|f(t)|_{H}^{2}e^{-2\rho t}\mbox{ d}t<\infty\right\} 
\]
 of $H$-valued functions and equip it with the following inner product
\[
\langle f|g\rangle_{\rho,0}\coloneqq\intop_{\mathbb{R}}\langle f(t)|g(t)\rangle_{H}e^{-2\rho t}\mbox{ d}t.
\]
We define the derivative $\partial_{0,\rho}$ on $H_{\rho,0}(\mathbb{R};H)$
as the closure of 
\begin{align*}
\interior C_{\infty}(\mathbb{R};H)\subseteq H_{\rho,0}(\mathbb{R};H) & \to H_{\rho,0}(\mathbb{R};H)\\
\phi & \mapsto\phi'.
\end{align*}

If the choice of $\rho$ is clear from the context, we will omit the
index $\rho$ and just write $\partial_{0}$.\end{defn}
\begin{rem}
(a) The subscript $0$ in the operator $\partial_{0}$ should remind
of the ``zero'th'' coordinate, which in Physics' literature often
plays the role of the time-derivative. We shall adopt this custom
here. 

(b) By definition the operator $e^{-\rho m}:H_{\rho,0}(\mathbb{R};H)\to L^{2}(\mathbb{R};H),f\mapsto(t\mapsto e^{-\rho t}f(t))$
is unitary. Consequently $\mathcal{L}_{\rho}\coloneqq\mathcal{F}e^{-\rho m}:H_{\rho,0}(\mathbb{R};H)\to L^{2}(\mathbb{R};H)$
is unitary, where we denote by $\mathcal{F}:L^{2}(\mathbb{R};H)\to L^{2}(\mathbb{R};H)$
the unitary Fourier-transform given by 
\[
\left(\mathcal{F}\phi\right)(x)\coloneqq\frac{1}{\sqrt{2\pi}}\intop_{\mathbb{R}}e^{-\i xy}\phi(y)\mbox{ d}y\quad(x\in\mathbb{R},\phi\in\interior C_{\infty}(\mathbb{R};H)).
\]
One can show that
\begin{equation}
\partial_{0}=\mathcal{L}_{\rho}^{\ast}\left(\i m+\rho\right)\mathcal{L}_{\rho},\label{eq:unitary_equiv}
\end{equation}
where $m:D(m)\subseteq L^{2}(\mathbb{R};H)\to L^{2}(\mathbb{R};H),f\mapsto\left(t\mapsto tf(t)\right)$
with maximal domain, yielding that $m$ is self-adjoint, and, by \eqref{eq:unitary_equiv},
proving that $\partial_{0}$ has is continuously invertible. Moreover,
from \eqref{eq:unitary_equiv} we read off that $\sigma(\partial_{0}),$
the spectrum of $\partial_{0}$, coincides with $i\left[\mathbb{R}\right]+\rho$.

Let $\rho_{0}>0$ and $M:B_{\mathbb{C}}\left(\frac{1}{2\rho_{0}},\frac{1}{2\rho_{0}}\right)\to L(H)$
be analytic such that 
\begin{equation}
\Re\langle z^{-1}M(z)x|x\rangle_{H}\geq0\label{eq:nn}
\end{equation}
for every $z\in B_{\mathbb{C}}\left(\frac{1}{2\rho_{0}},\frac{1}{2\rho_{0}}\right),\, x\in H.$
Defining 
\[
M\left(\frac{1}{\i m+\rho}\right):L^{2}(\mathbb{R};H)\to L^{2}(\mathbb{R};H),\, f\mapsto\left(t\mapsto M\left(\frac{1}{\i t+\rho}\right)f(t)\right)
\]
for $\rho>\rho_{0}$ we can generate operator-valued multipliers and
via $M(\partial_{0}^{-1})\coloneqq\mathcal{L}_{\rho}^{\ast}M\left(\frac{1}{\i m+\rho}\right)\mathcal{L}_{\rho}$
corresponding operator-valued functions of $\partial_{0}^{-1}$. We
shall loosely refer to such operators $M\left(\partial_{0}^{-1}\right)$
as material law operators.
\end{rem}
We quote the following well-posedness result.
\begin{thm}[{\cite[Theorem 6.2.5]{PDE_DeGruyter}, \cite[Solution Theory]{Pi2009-1}}]
\label{thm:sol_th}Let $H$ be a Hilbert space and $A:D(A)\subseteq H\to H$
a skew-selfadjoint linear operator. Moreover, let $\rho_{0}>0$ and
\textup{$M:B_{\mathbb{C}}\left(\frac{1}{2\rho_{0}},\frac{1}{2\rho_{0}}\right)\to L(H)$}
be analytic such that there exists $c>0$ with 
\begin{equation}
\Re\langle z^{-1}M(z)x|x\rangle_{H}\geq c|x|_{H}^{2}\label{eq:pos-def}
\end{equation}
for every $z\in B_{\mathbb{C}}\left(\frac{1}{2\rho_{0}},\frac{1}{2\rho_{0}}\right),\, x\in H.$
Then for each $\rho>\rho_{0}$ the operator $\left(\overline{\partial_{0}M(\partial_{0}^{-1})+A}\right)^{-1}$
is boundedly invertible on $H_{\rho,0}(\mathbb{R};H)$ and causal.\end{thm}
\begin{rem}
(a) In many applications, $M$ has the simple form $M(z)=M_{0}+zM_{1}$
for some $M_{0},M_{1}\in L(H).$ Condition \eqref{eq:pos-def} can
then for example be achieved for $\rho_{0}$ large enough, if $M_{0}$
is selfadjoint and there exist $c_{0},c_{1}>0$ such that 
\[
\langle M_{0}x|x\rangle_{H}\geq c_{0}|x|_{H}^{2}\mbox{ and }\Re\langle M_{1}y|y\rangle_{H}\geq c_{1}|y|_{H}^{2}
\]
for every $x$ belonging to $M_{0}[H]$, the range of $M_{0}$, and
every $y$ from $N(M_{0})$, the nullspace of $M_{0}$. The corresponding
material law operator is then 
\[
M(\partial_{0}^{-1})=M_{0}+\partial_{0}^{-1}M_{1}.
\]

(b) For the interested reader, we shall also mention possible generalization
of Theorem \ref{thm:sol_th} to non-autonomous equations (\cite{Waurick2013JEE_Non-auto,Waurick2014MMAS_Non-auto})
or non-linear equations (\cite{Trostorff2014a,Trostorff2012})
\end{rem}

\subsection{The Guyer-Krumhansl model of heat conduction.}

The Guyer-Krumhansl model of heat conduction (see e.g.~\cite{PhysRev.133.A1411,PhysRev.148.766,PhysRev.148.778})
consists of two equations: A first equation, the balance equation,
relates the heat $\theta$ to the heat flux $q$ in the way that
\[
\rho c\partial_{0}\theta+\dive q=h,
\]
where $\rho$ and $c$ are certain material parameters and $h$ is
a given source term. The difference to the classical equations of
thermodynamics is the following alternative to Fourier's law:
\begin{align*}
\left(1+\tau_{0}\partial_{0}\right)q & =-\kappa\grad\theta+\mu_{1}\Delta q+\mu_{2}\grad\dive q,
\end{align*}
where $\tau_{\text{0}}$, $\kappa,$ $\mu_{1},$ $\mu_{2}$ are real
numbers modeling the material's properties. We mention that the choices
$\tau_{0}=\mu_{1}=\mu_{2}=0$ recover the standard Fourier law. In
any case, $\kappa$ is assumed to be strictly positive. With this
observation, we may reformulate the modified Fourier's law as
\[
\left(\tau_{0}\kappa^{-1}\partial_{0}+\kappa^{-1}-\mu_{1}\kappa^{-1}\Delta-\mu_{2}\kappa^{-1}\grad\dive\right)q=-\grad\theta.
\]
Here we emphasize that here $\Delta$ -- as does $\grad\dive$ --
acts on vectors with $3$ components. Summarizing, we end up with
a system of the form 
\[
\left(\begin{array}{cc}
\rho c\partial_{0} & \quad\dive\\
\grad & \quad\left(\tau_{0}\kappa^{-1}\partial_{0}+\kappa^{-1}-\mu_{1}\kappa^{-1}\Delta-\mu_{2}\kappa^{-1}\grad\dive\right)
\end{array}\right)\left(\begin{array}{c}
\theta\\
q
\end{array}\right)=\left(\begin{array}{c}
h\\
0
\end{array}\right)
\]
and thus, the block operator matrix to be studied reads as follows:
\[
\left(\begin{array}{cc}
\rho c\partial_{0} & \quad\dive\\
\grad & \quad\left(\tau_{0}\kappa^{-1}\partial_{0}+\kappa^{-1}-\mu_{1}\kappa^{-1}\Delta-\mu_{2}\kappa^{-1}\grad\dive\right)
\end{array}\right).
\]
In the following we show that the latter block operator matrix is
of the form of abstract evolutionary equations as mentioned in the
introduction. Rearranging terms and separating the spatial derivatives
from the time-derivative, the operator becomes
\[
\partial_{0}\left(\begin{array}{cc}
\rho c & 0\\
0 & \tau_{0}\kappa^{-1}
\end{array}\right)+\left(\begin{array}{cc}
0 & 0\\
0 & \kappa^{-1}
\end{array}\right)+\left(\begin{array}{cc}
0 & \quad\dive\\
\grad & \quad\left(-\mu_{1}\kappa^{-1}\Delta-\mu_{2}\kappa^{-1}\grad\dive\right)
\end{array}\right).
\]
In order to obtain easily accessible boundary conditions leading to
the well-posedness of the latter equation, we give a different representation
of the term $\left(-\mu_{1}\kappa^{-1}\Delta-\mu_{2}\kappa^{-1}\grad\dive\right)$
on the basis of our theory of abstract $\grad$-$\dive$-systems.
Beforehand, however, we will show the following lemma of combinatoric
nature:
\begin{lem}
\label{lem:combinatorical}Let $\mu_{2}\in\mathbb{R}$, $\mu_{1},\kappa\in\mathbb{R}_{>0}$
with $\mu_{1}>-\mu_{2}$. Then there exists $C=C^{*}\in L(\mathbb{C}^{3\times3},\mathbb{C}^{3\times3})$
strictly positive definite such that for all $\phi\in C_{\infty}\left(\overline{\Omega}\right)^{3}$
we have%
\footnote{Recall that here $\Delta$ is the Laplacian acting as $\left(\begin{array}{ccc}
\Delta & 0 & 0\\
0 & \Delta & 0\\
0 & 0 & \Delta
\end{array}\right)$ and that $\dive$ on the right-hand side is the row-wise divergence
and the $\grad$ on the right-hand side is the Jacobian of $\phi.$ %
}
\[
\left(\kappa^{-1}\mu_{1}\Delta+\kappa^{-1}\mu_{2}\grad\dive\right)\phi=\dive C\grad\phi.
\]
\end{lem}
\begin{proof}
Without restriction assume that $\kappa=1$. As an ansatz, take 
\begin{equation}
C=\alpha_{0}\mathrm{sym}_{0}+\alpha_{1}\mathbb{P}+\alpha_{2}\mathrm{skew}\label{eq:ansatz_for_C}
\end{equation}
for $\alpha_{0},\alpha_{1},\alpha_{2}\in\oi0\infty$ to be determined
later on, where
\[
\mathrm{sym}_{0},\mathrm{sym},\mathbb{P},\mathrm{skew}\colon\mathbb{C}^{3\times3}\to\mathbb{C}^{3\times3}
\]
are defined by 
\begin{align*}
\mathbb{P} & \coloneqq\frac{1}{3}\mathrm{trace}^{*}\mathrm{trace},\\
\mathrm{sym}\:\sigma & \coloneqq\frac{1}{2}\left(\sigma+\sigma^{\top}\right)\quad\left(\sigma\in\mathbb{C}^{3\times3}\right),\\
\mathrm{skew}\:\sigma & \coloneqq\frac{1}{2}\left(\sigma-\sigma^{\top}\right)\quad\left(\sigma\in\mathbb{C}^{3\times3}\right),\\
\mathrm{sym}_{0} & \coloneqq\left(1-\mathbb{P}\right)\mathrm{sym}=\mathrm{sym}\left(1-\mathbb{P}\right)
\end{align*}
with $\mathrm{trace}:\mathbb{C}^{3\times3}\to\mathbb{C},\sigma\mapsto\sum_{i=1}^{3}\sigma_{ii}$.
Observe that 
\begin{align*}
\mathrm{trace}^{*}:\mathbb{C} & \to\mathbb{C}^{3\times3},\\
z & \mapsto z\mathbb{I}_{3\times3},
\end{align*}
where $\mathbb{I}_{3\times3}$ denotes the identity matrix in $\mathbb{C}^{3\times3}$
and, hence, $\mathbb{P}\sigma=\frac{\trace\sigma}{3}\left(\begin{array}{ccc}
1 & 0 & 0\\
0 & 1 & 0\\
0 & 0 & 1
\end{array}\right)$. Observing that $\trace\grad\phi=\dive\phi,$ we get that
\begin{align*}
\dive\mathrm{sym}_{0}\grad\phi & =\dive\left(\frac{1}{2}\left(\partial_{i}\phi_{j}+\partial_{j}\phi_{i}-\frac{2}{3}\delta_{ij}\dive\phi\right)_{i,j\in\{1,2,3\}}\right)\\
 & =\left(\sum_{j=1}^{3}\partial_{j}\frac{1}{2}\left(\partial_{i}\phi_{j}+\partial_{j}\phi_{i}-\delta_{ij}\frac{2}{3}\dive\phi\right)\right)_{i\in\{1,2,3\}}\\
 & =\frac{1}{2}\Delta\phi+\frac{1}{2}\grad\dive\phi-\frac{1}{3}\grad\dive\phi\\
 & =\frac{1}{2}\left(\Delta\phi-\grad\dive\phi\right)+\frac{2}{3}\grad\dive\phi\\
 & =-\frac{1}{2}\curl\curl\phi+\frac{2}{3}\grad\dive\phi
\end{align*}
and
\begin{align*}
\dive\mathbb{P}\grad\phi & =\dive\mathbb{P}\:\left(\partial_{j}\phi_{i}\right)_{i,j\in\{1,2,3\}}\\
 & =\dive\left(\frac{1}{3}\delta_{ij}\dive\phi\right)_{i,j\in\{1,2,3\}}\\
 & =\left(\sum_{j=1}^{3}\partial_{j}\frac{1}{3}\delta_{ij}\dive\phi\right)_{i\in\{1,2,3\}}\\
 & =\frac{1}{3}\grad\dive\phi.
\end{align*}
Similarly, we compute
\begin{align*}
\dive\mathrm{skew}\grad\phi & =\dive\mathrm{skew}\left(\partial_{k}\phi_{j}\right)_{j,k\in\{1,2,3\}}\\
 & =\frac{1}{2}\dive\left(\partial_{k}\phi_{j}-\partial_{j}\phi_{k}\right)_{j,k\in\{1,2,3\}}\\
 & =\frac{1}{2}\left(\sum_{k=1}^{3}\partial_{k}\left(\partial_{k}\phi_{j}-\partial_{j}\phi_{k}\right)\right)_{j\in\{1,2,3\}}\\
 & =\frac{1}{2}\left(\Delta-\grad\dive\right)\phi=-\frac{1}{2}\curl\curl\phi.
\end{align*}
{} Thus, choosing $C$ as in \eqref{eq:ansatz_for_C}, we get that
\begin{align*}
\dive C\grad\phi & =\dive\left(\alpha_{0}\mathrm{sym}_{0}+\alpha_{1}\mathbb{P}+\alpha_{2}\mathrm{skew}\right)\grad\phi\\
 & =\alpha_{0}\dive\mathrm{sym}_{0}\grad\phi+\alpha_{1}\dive\mathbb{P}\grad\phi+\alpha_{2}\dive\mathrm{skew}\grad\phi\\
 & =\alpha_{0}\left(-\frac{1}{2}\curl\curl\phi+\frac{2}{3}\grad\dive\phi\right)+\alpha_{1}\frac{1}{3}\grad\dive\phi-\alpha_{2}\frac{1}{2}\curl\curl\phi\\
 & =-\frac{\alpha_{0}+\alpha_{2}}{2}\curl\curl\phi+\frac{2\alpha_{0}+\alpha_{1}}{3}\grad\dive\phi\\
 & =\frac{\alpha_{0}+\alpha_{2}}{2}\left(-\curl\curl+\grad\dive\right)+\left(\frac{2\alpha_{0}+\alpha_{1}}{3}-\frac{\alpha_{0}+\alpha_{2}}{2}\right)\grad\dive\\
 & =\frac{\alpha_{0}+\alpha_{2}}{2}\Delta+\frac{\alpha_{0}+2\alpha_{1}-3\alpha_{2}}{6}\grad\dive.
\end{align*}
Comparing coefficients, we get 
\begin{align*}
\frac{\alpha_{0}+\alpha_{2}}{2} & =\mu_{1},\\
\frac{\alpha_{0}+2\alpha_{1}-3\alpha_{2}}{6} & =\mu_{2}.
\end{align*}
Now, strict positive definiteness of $C$ is equivalent to%
\footnote{Note that $\mathbb{P},\,\mathrm{sym}_{0}$ and $\mathrm{skew}$ are
projections on pairwise orthogonal subspaces of $\mathbb{C}^{3\times3}$.%
} $\alpha_{1},\alpha_{2},\alpha_{3}>0$. Introducing $\lambda\coloneqq\alpha_{0}-3\alpha_{2}\in\mathbb{R}$,
we get
\begin{align*}
\alpha_{1} & =3\left(\mu_{2}-\frac{\lambda}{6}\right),\\
\alpha_{0} & =\frac{3}{2}\left(\mu_{1}+\frac{\lambda}{6}\right),\\
\alpha_{2} & =\frac{1}{2}\left(\mu_{1}-\frac{\lambda}{2}\right).
\end{align*}
If $\mu_{2}\geq0$ then $\lambda=-3\mu_{1}<0$ is a possible choice
for the latter set of equations to obtain strict positive definiteness
of $C$. If $\mu_{2}<0$ then $\lambda=3(\mu_{2}-\mu_{1})$ leads
to $\mu_{1}+\frac{\lambda}{6}=\mu_{1}+\frac{\mu_{2}-\mu_{1}}{2}=\frac{\mu_{1}+\mu_{2}}{2}>0$,
$\mu_{2}-\frac{\lambda}{6}=\mu_{2}-\frac{\mu_{2}-\mu_{1}}{2}=\frac{\mu_{1}+\mu_{2}}{2}>0$
as well as $\mu_{1}-\frac{\lambda}{2}=\mu_{1}-\frac{3}{2}\left(\mu_{2}-\mu_{1}\right)=\frac{5\mu_{1}-3\mu_{2}}{2}>0$,
which also implies that $C$ is strictly positive definite. 
\end{proof}
Lemma \ref{lem:combinatorical} together with the previous remark
provides a way for formulating the Guyer-Krumhansl model for heat
conduction in the framework of evolutionary equations. This, in turn,
results in the following well-posedness theorem, where we have for
sake of definiteness chosen a specific case of boundary conditions:
\begin{thm}
Let $\tau_{0},\rho,\kappa\in\mathbb{R}_{>0},$ $\mu_{1},\mu_{2}\in\mathbb{R}$
with $\mu_{2}>-\mu_{1}$, $\Omega\subseteq\mathbb{R}^{3}$ open and
$C$ as in Lemma \ref{lem:combinatorical}. Let $c\in L(L^{2}(\Omega))$
be selfadjoint and strictly positive definite and $A\colon D(A)\subseteq L^{2}\left(\Omega\right)^{3}\oplus L^{2}\left(\Omega\right)\oplus L^{2}\left(\Omega\right)^{3\times3}\to L^{2}\left(\Omega\right)^{3}\oplus L^{2}\left(\Omega\right)\oplus L^{2}\left(\Omega\right)^{3\times3}$
with
\[
A\psi=\left(\begin{array}{cc}
0 & \left(\begin{array}{cc}
\grad & \quad-\dive\end{array}\right)\\
\left(\begin{array}{c}
\interior\dive\\
-\interior\grad
\end{array}\right) & 0
\end{array}\right)\psi
\]
for all 
\[
\psi\in D(A)\coloneqq\left\{ \left.\left(\begin{array}{c}
\psi_{1}\\
\psi_{2}\\
\psi_{3}
\end{array}\right)\in\left(\begin{array}{c}
L^{2}\left(\Omega\right)^{3}\\
L^{2}\left(\Omega\right)\\
L^{2}\left(\Omega\right)^{3\times3}
\end{array}\right)\,\right|\,\psi_{1}\in D(\interior\grad),\grad\psi_{2}-\dive\psi_{3}\in L^{2}(\Omega)^{3}\right\} .
\]
Then the (closure of the) operator 
\[
\mathcal{T}\coloneqq\partial_{0}\left(\begin{array}{ccc}
\tau_{0}\kappa^{-1} & 0 & 0\\
0 & \rho c & 0\\
0 & 0 & 0
\end{array}\right)+\left(\begin{array}{ccc}
\kappa^{-1} & 0 & 0\\
0 & 0 & 0\\
0 & 0 & C^{-1}
\end{array}\right)+\left(\begin{array}{cc}
0 & \left(\begin{array}{cc}
\grad & \quad-\dive\end{array}\right)\\
\left(\begin{array}{c}
\interior\dive\\
-\interior\grad
\end{array}\right) & 0
\end{array}\right)
\]
is continuously invertible in $H_{\rho,0}\left(\mathbb{R};L^{2}\left(\Omega\right)^{3}\oplus L^{2}\left(\Omega\right)\oplus L^{2}\left(\Omega\right)^{3\times3}\right)$
for sufficiently large $\rho>0$. \end{thm}
\begin{proof}
First of all we observe that $A$ is skew-selfadjoint. Indeed, $\left(\begin{array}{c}
\interior\dive\\
-\interior\grad
\end{array}\right):\interior H_{1}(\Omega)^{3}\subseteq L^{2}\left(\Omega\right)^{3}\to L^{2}(\Omega)\oplus L^{2}(\Omega)^{3\times3}$ generates an abstract $\grad$-$\dive$-system (use Proposition \ref{prop:standard_grad-div}
with $C_{0}=-\interior\grad$, $D(C_{0})=\interior H_{1}(\Omega)^{3}$
and $C_{1}=\interior\dive$) %
\footnote{As a subtlety not to be missed, here we have taken $\grad:L^{2}\left(\Omega\right)\to\left(\interior H_{1}(\Omega)^{3}\right)^{\prime}$
acting on scalars as the dual $\left(-\interior\dive\right)^{\diamond}$
of $-\interior\dive\colon\interior H_{1}(\Omega)^{3}\to L^{2}(\Omega)$.
This is in contrast to the situation in acoustics, where only the
vanishing of the normal component on the boundary (or a generalization
thereof) is imposed for the domain of $-\interior\dive$. The tensor
divergence occurring here is correspondingly $\dive:L^{2}\left(\Omega\right)^{3\times3}\to\left(\interior H_{1}(\Omega)^{3}\right)^{\prime}$
, the dual of $-\interior\grad$ on vector fields. %
} and the skew-selfadjointness of $A$ follows from Theorem \ref{thm:AGDS-adjoint}.
Therefore, the invertibility of $\overline{\mathcal{T}}$ is guaranteed
by the strict positive definiteness of $c$ and $C$ by Theorem \ref{thm:sol_th}. 
\end{proof}
Thus, for $h\in H_{\rho,0}\left(\mathbb{R};L^{2}\left(\Omega\right)\right)$
and $u=(u_{1},u_{2},u_{3})$ we have 
\[
\left(\partial_{0}\left(\begin{array}{ccc}
\tau_{0}\kappa^{-1} & 0 & 0\\
0 & \rho c & 0\\
0 & 0 & 0
\end{array}\right)+\left(\begin{array}{ccc}
\kappa^{-1} & 0 & 0\\
0 & 0 & 0\\
0 & 0 & C^{-1}
\end{array}\right)+\left(\begin{array}{cc}
0 & \left(\begin{array}{cc}
\grad & \quad-\dive\end{array}\right)\\
\left(\begin{array}{c}
\interior\dive\\
-\interior\grad
\end{array}\right) & 0
\end{array}\right)\right)\left(\begin{array}{c}
u_{1}\\
u_{2}\\
u_{3}
\end{array}\right)=\left(\begin{array}{c}
0\\
h\\
0
\end{array}\right).
\]
This is 
\begin{align*}
\left(\tau_{0}\kappa^{-1}\partial_{0}u_{1}+\kappa^{-1}u_{1}\right)+\grad u_{2}-\dive u_{3} & =0,\\
\rho c\partial_{0}u_{2}+\dive u_{1} & =h,\\
C^{-1}u_{3}-\grad u_{1} & =0.
\end{align*}
Substituting -- as a formal calculation -- the third equation into
the first, we get
\[
\left(\tau_{0}\kappa^{-1}\partial_{0}u_{1}+\kappa^{-1}u_{1}\right)+\grad u_{2}-\dive C\grad u_{1}=0,
\]
which, by Lemma \ref{lem:combinatorical} gives the original form
of the system back
\[
\left(\tau_{0}\kappa^{-1}\partial_{0}u_{1}+\kappa^{-1}u_{1}\right)=\kappa^{-1}\mu_{1}\Delta u_{1}+\kappa^{-1}\mu_{2}\grad\dive u_{1}-\grad u_{2}.\tag*{\qedhere}
\]

\begin{rem}
~
\begin{enumerate}
\item Note that the role of temperature is played by the second block component
$\psi_{2}$ whereas $\psi_{1}$ is the heat flux in this model.
\item The system reveals appropriate possible generalizations for the material
law operator. For example
\[
M\left(\partial_{0}^{-1}\right)=\left(\begin{array}{ccc}
\kappa_{0}^{-1} & 0 & 0\\
0 & c_{0} & 0\\
0 & 0 & 0
\end{array}\right)+\partial_{0}^{-1}\left(\begin{array}{ccc}
\kappa^{-1} & 0 & 0\\
0 & 0 & 0\\
0 & 0 & C^{-1}
\end{array}\right)
\]
where $\kappa,\:\kappa_{0},\, c_{0},\, C$ are a strictly positive
definite \emph{operators} acting in a matching $L^{2}\left(\Omega\right)$-setting
rather than just numbers. 
\item In proper tensorial terms the spatial operator should actually be
written as
\[
\left(\begin{array}{cc}
0 & \left(\begin{array}{cc}
\dive\trace^{\ast} & \quad-\dive\end{array}\right)\\
\left(\begin{array}{c}
\trace\interior\grad\\
-\interior\grad
\end{array}\right) & \left(\begin{array}{cc}
0 & 0\\
0 & 0
\end{array}\right)
\end{array}\right)=\left(\begin{array}{cc}
0 & \dive\left(\begin{array}{cc}
\trace^{\ast} & \quad-1\end{array}\right)\\
\left(\begin{array}{c}
\trace\\
-1
\end{array}\right)\interior\grad & \left(\begin{array}{cc}
0 & 0\\
0 & 0
\end{array}\right)
\end{array}\right).
\]

\item If preferred we may specialize to the symmetric tensor case, compare
\cite{0953-8984-7-7-025,valenti1997heat}, and use 
\[
\left(\begin{array}{cc}
0 & \left(\begin{array}{cc}
\grad & -\dive\iota_{\mathrm{sym}}\end{array}\right)\\
\left(\begin{array}{c}
\interior\dive\\
-\iota_{\mathrm{sym}}^{*}\interior\grad
\end{array}\right) & \left(\begin{array}{cc}
0 & 0\\
0 & 0
\end{array}\right)
\end{array}\right)
\]
or, according to item 3, 
\[
\left(\begin{array}{cc}
0 & \dive\iota_{\mathrm{sym}}\left(\begin{array}{cc}
\mathrm{trace}^{*} & -1\end{array}\right)\\
\left(\begin{array}{c}
\mathrm{trace}\\
-1
\end{array}\right)\iota_{\mathrm{sym}}^{*}\interior\grad & \left(\begin{array}{cc}
0 & 0\\
0 & 0
\end{array}\right)
\end{array}\right)
\]
instead of the spatial operator above. Here we denote by $\iota_{\mathrm{sym}}:L_{\mathrm{sym}}^{2}(\Omega)^{3\times3}\to L^{2}(\Omega)^{3\times3}$
the canonical embedding of 
\[
L_{\mathrm{sym}}^{2}(\Omega)^{3\times3}\coloneqq\left\{ (f_{ij})_{i,j\in\{1,2,3\}}\in L^{2}(\Omega)^{3\times3}\,|\, f_{ij}=f_{ji}\quad(i,j\in\{1,2,3\})\right\} .
\]
In consequence, $\iota_{\mathrm{sym}}^{\ast}:L^{2}(\Omega)^{3\times3}\to L_{\mathrm{sym}}^{2}(\Omega)^{3\times3}$
is the orthogonal projection onto $L_{\mathrm{sym}}^{2}(\Omega)^{3\times3}.$
\end{enumerate}
\end{rem}

\subsection{\label{sub:wave}A class of dynamic boundary conditions.}

In this section, we shall analyze a class of dynamic boundary conditions
related to the wave equation. For this let $\Omega\subseteq\mathbb{R}^{3}$
be open and bounded with Lipschitz boundary $\Gamma\coloneqq\partial\Omega$
and unit outward normal field $n$. Further, we shall assume that
$\Gamma$ is a manifold, where it is possible to define the co-variant
derivative $\grad_{\Gamma}$ as an operator from $L^{2}(\Gamma)$
to $L_{\tau}^{2}(\Gamma)\coloneqq L^{2}\left(\Gamma\right)^{2}$.
By $\gamma\colon H_{1}(\Omega)\to L^{2}(\Gamma),\phi\mapsto\phi|_{\Gamma}$
we denote the trace mapping. Consider the operator equation 
\begin{equation}
\left(\partial_{0}M_{0}+M_{1}+A\right)\left(\begin{array}{c}
p\\
\left(\begin{array}{c}
v\\
\eta_{1}\\
\eta_{2}
\end{array}\right)
\end{array}\right)=\left(\begin{array}{c}
f\\
\left(\begin{array}{c}
g\\
h_{1}\\
h_{2}
\end{array}\right)
\end{array}\right)\in H_{\rho,0}\left(\mathbb{R};L^{2}(\Omega)\oplus L^{2}(\Omega)\oplus L^{2}(\Omega)\oplus L_{\tau}^{2}(\Gamma)\right)\label{eq:acoustic-bdy}
\end{equation}
for $\rho$ sufficiently large and with 
\[
A=\left(\begin{array}{cc}
\left(0\right) & -\left(\begin{array}{c}
\grad\\
\gamma\\
\grad_{\Gamma}\gamma
\end{array}\right)^{*}\\
\left(\begin{array}{c}
\grad\\
\gamma\\
\grad_{\Gamma}\gamma
\end{array}\right) & \left(\begin{array}{ccc}
0 & 0 & 0\\
0 & 0 & 0\\
0 & 0 & 0
\end{array}\right)
\end{array}\right),
\]
where $M_{0}$ is selfadjoint and strictly positive definite on its
range, while $M_{1}$ is strictly positive definite on the null space
of $M_{0}$. Note that by the general solution theory Theorem \ref{thm:sol_th}
well-posedness of the system just discussed is not an issue here,
if we can ensure that $A$ is skew-selfadjoint. For this, we shall
have a closer look at the operator $A$. Defining for this section
\begin{align*}
X_{0} & \coloneqq L^{2}(\Omega),\\
X_{1} & \coloneqq\left(D(\grad_{\Gamma}\gamma),\sqrt{|\grad_{\Gamma}\gamma\cdot|_{L_{\tau}^{2}(\Gamma)}^{2}+|\cdot|_{H_{1}(\Omega)}^{2}}\:\right),
\end{align*}
we realize that \eqref{eq:acoustic-bdy} can be treated as an abstract
$\grad$--$\dive$ system:
\begin{prop}
\label{prop:Acoustic_grad_div}The operator
\[
C\coloneqq\left(\begin{array}{c}
\grad\\
\gamma\\
\grad_{\Gamma}\gamma
\end{array}\right):D\left(\grad_{\Gamma}\gamma\right)\subseteq X_{0}\to L^{2}(\Omega)\oplus L^{2}(\Gamma)\oplus L_{\tau}^{2}(\Gamma)
\]
 generates an abstract $\grad$--$\dive$ system.\end{prop}
\begin{proof}
At first, note that $X_{1}\hookrightarrow X_{0}$ with dense and continuous
embedding, yielding in particular that $C$ is densely defined. Moreover,
as $\gamma\colon H_{1}(\Omega)\to L^{2}(\Gamma)$ is continuous and
$\grad_{\Gamma}:H_{1}(\Gamma)\subseteq L^{2}(\Gamma)\to L_{\tau}^{2}(\Gamma)$
is closed, the operator $\grad_{\Gamma}\gamma\colon D\left(\grad_{\Gamma}\gamma\right)\subseteq H_{1}(\Omega)\to L_{\tau}^{2}(\Gamma)$
is closed as well. Thus, $C$ generates an abstract $\grad$--$\dive$
system by Proposition \ref{prop:standard_grad-div} for the choices
$C_{0}=\grad_{\Gamma}\gamma$, $C_{1}=\gamma$ and $C_{2}=\grad$.
\end{proof}
We shall now focus on the equations satisfied by $\left(\begin{array}{c}
p\\
\left(\begin{array}{c}
v\\
\eta_{1}\\
\eta_{2}
\end{array}\right)
\end{array}\right)$, especially the boundary conditions, which are encoded in the domain
of $A$. 
\begin{lem}
\label{lem:What is gradgg}Let $\grad_{\Gamma}\gamma:X_{1}\to L^{2}\left(\Gamma\right)$.
Then $\left(\grad_{\Gamma}\gamma\right)^{\diamond}=\gamma^{'}\grad_{\Gamma}^{\diamond}$
with $\grad_{\Gamma}\colon H_{1}(\Gamma)\to L_{\tau}^{2}(\Gamma)$
and $\gamma\colon X_{1}\to H_{1}(\Gamma)$ where $\gamma'\colon H_{1}(\Gamma)'\to X_{1}'$
denotes the dual operator of $\gamma$. \end{lem}
\begin{proof}
For $\phi\in L_{\tau}^{2}(\Gamma)$, we compute with $\psi\in D(\grad_{\Gamma}\gamma)=X_{1}$:
\begin{align*}
\left(\left(\grad_{\Gamma}\gamma\right)^{\diamond}\phi\right)(\psi) & =\langle\phi|\grad_{\Gamma}\gamma\psi\rangle_{L_{\tau}^{2}(\Gamma)}\\
 & =\left(\grad_{\Gamma}^{\diamond}\phi\right)(\gamma\psi)\\
 & =\left(\gamma'\grad_{\Gamma}^{\diamond}\phi\right)(\psi).\tag*{{\qedhere}}
\end{align*}
\end{proof}
\begin{rem}
\label{rem:grad-dach, div} For $\dive_{\Gamma}\coloneqq-\grad_{\Gamma}^{*}$,
we have the inclusion $-\gamma'\dive_{\Gamma}\subseteq\left(\grad_{\Gamma}\gamma\right)^{\diamond}.$
In fact, this follows from the previous lemma and the inclusion $-\dive_{\Gamma}\subseteq\grad_{\Gamma}^{\diamond}$,
which holds by Lemma \ref{lem:adjoint-dual}. %
\end{rem}
\begin{lem}
\label{lem:boundary equation}Let $v\in D(\dive)$, $\eta_{1}\in L^{2}(\Gamma)$,
$\eta_{2}\in H_{1}(\Gamma)'$, where $H_{1}(\Gamma)=D(\grad_{\Gamma})$
equipped with the graph norm. Consider the operators $\grad\colon X_{1}\to L^{2}\left(\Omega\right)^{3},$
$\gamma_{1}\colon X_{1}\to L^{2}(\Gamma),x_{1}\mapsto\gamma x_{1}$
and $\gamma_{2}\colon X_{1}\to H_{1}(\Gamma),x_{1}\mapsto\gamma x_{1}$.
Then $-\dive v=\grad^{\diamond}v+\gamma_{1}^{\diamond}\eta_{1}+\gamma_{2}'\eta_{2}$
if and only if%
\footnote{Note that a-priori $n\cdot v$ is only defined as a functional in
$H_{1/2}\left(\Gamma\right)^{\prime}$, where $H_{1/2}(\Gamma)=\gamma[H_{1}(\Omega)]$.
Indeed, $n\cdot v$ is the element on $H_{1/2}(\Gamma)'$ such that
for all $\psi\in H_{1}(\Omega)$ the equation
\[
\left(n\cdot v\right)(\gamma\psi)=\langle\dive v|\psi\rangle_{L^{2}(\Omega)}+\langle v|\grad\psi\rangle_{L^{2}(\Omega)^{3}}
\]
is satisfied.%
} $n\cdot v+\eta_{1}+\eta_{2}=0$ on $L^{2}\left(\Gamma\right)$.\end{lem}
\begin{proof}
$(\Rightarrow)$ Assume that $-\dive v=\grad^{\diamond}v+\gamma_{1}^{\diamond}\eta_{1}+\gamma_{2}'\eta_{2}$.
Then, for $\psi\in X_{1}$ we have that $\gamma\psi=\gamma_{1}\psi=\gamma_{2}\psi\in H_{1}(\Gamma)\subseteq H_{1/2}(\Gamma)\subseteq L^{2}(\Gamma)$
and we compute
\begin{align*}
\left(n\cdot v+\eta_{1}+\eta_{2}\right)(\gamma\psi) & =\langle\dive v|\psi\rangle_{L^{2}\left(\Omega\right)}+\langle v|\grad\psi\rangle_{L^{2}\left(\Omega\right)^{3}}+\langle\eta_{1}|\gamma_{1}\psi\rangle_{L^{2}\left(\Gamma\right)}+\eta_{2}(\gamma\psi)\\
 & =\langle\dive v|\psi\rangle_{L^{2}\left(\Omega\right)}+\left(\grad^{\diamond}v+\gamma_{1}^{\diamond}\eta_{1}+\gamma_{2}'\eta_{2}\right)(\psi)\\
 & =\langle\dive v|\psi\rangle_{L^{2}(\Omega)}-\langle\dive v|\psi\rangle_{L^{2}(\Omega)}\\
 & =0.
\end{align*}
Thus, the assertion follows, if we show that $\gamma[X_{1}]$ is dense
in $L^{2}(\Gamma)$. For this, we note that $C_{\infty}(\overline{\Omega})$
is dense in $H_{1}(\Omega)$ since $\Omega$ is bounded and has Lipschitz
boundary. Moreover, $\gamma:H_{1}(\Omega)\to H_{1/2}(\Gamma)$ is
continuous and onto and hence, $\gamma[C_{\infty}(\overline{\Omega})]$
is dense in $H_{1/2}(\Omega)$ and thus, dense in $L^{2}(\Gamma).$
Since $C_{\infty}(\overline{\Omega})\subseteq D(\grad_{\Gamma}\gamma)=X_{1}$
we derive the assertion. 

$(\Leftarrow)$ Assume now that $n\cdot v+\eta_{1}+\eta_{2}=0$ on
$L^{2}\left(\Gamma\right)$. Then we compute for $\psi\in X_{1}$
\begin{align*}
0 & =\langle n\cdot v+\eta_{1}+\eta_{2}|\gamma\psi\rangle_{L^{2}(\Gamma)}\\
 & =\langle\dive v|\psi\rangle_{L^{2}\left(\Omega\right)}+\langle v|\grad\psi\rangle_{L^{2}\left(\Omega\right)^{3}}+\langle\eta_{1}|\gamma_{1}\psi\rangle_{L^{2}\left(\Gamma\right)}+\eta_{2}(\gamma_{2}\psi)\\
 & =\langle\dive v|\psi\rangle_{L^{2}\left(\Omega\right)}+\left(\grad^{\diamond}v+\gamma_{1}^{\diamond}\eta_{1}+\gamma_{2}'\eta_{2}\right)(\psi)
\end{align*}
which gives the assertion. 
\end{proof}
With this lemma we can characterize the domain of $C^{\ast}$ with
$C\coloneqq\left(\begin{array}{c}
\grad\\
\gamma\\
\grad_{\Gamma}\gamma
\end{array}\right)$ in terms of an abstract boundary condition.
\begin{thm}
We have 
\[
\left(\begin{array}{c}
\grad\\
\gamma\\
\grad_{\Gamma}\gamma
\end{array}\right)^{*}\left(\begin{array}{c}
v\\
\eta_{1}\\
\eta_{2}
\end{array}\right)=-\dive v
\]
with 
\[
D\left(\left(\begin{array}{c}
\grad\\
\gamma\\
\grad_{\Gamma}\gamma
\end{array}\right)^{*}\right)=\left\{ (v,\eta_{1},\eta_{2})\in D\left(\dive\right)\times L^{2}(\Gamma)\times L_{\tau}^{2}(\Gamma)\,|\, n\cdot v+\eta_{1}+\grad_{\Gamma}^{\diamond}\eta_{2}=0\text{ on }L^{2}\left(\Gamma\right)\right\} .
\]
\end{thm}
\begin{proof}
In Proposition \ref{prop:Acoustic_grad_div}, we had already seen
that $C=\left(\begin{array}{c}
\grad\\
\gamma\\
\grad_{\Gamma}\gamma
\end{array}\right)$ generates an abstract $\grad$-$\dive$-system and thus, by Theorem
\ref{thm:AGDS-adjoint} we get $(v,\eta_{1},\eta_{2})\in D(C^{*})$
if and only if $\grad^{\diamond}v+\gamma^{\diamond}\eta_{1}+(\grad_{\Gamma}\gamma)^{\diamond}\eta_{2}\in L^{2}(\Omega)$.
Moreover, we clearly have
\[
\left(\begin{array}{c}
\interior\grad\\
0\\
0
\end{array}\right)\subseteq\left(\begin{array}{c}
\grad\\
\gamma\\
\grad_{\Gamma}\gamma
\end{array}\right)
\]
and hence, Corollary \ref{cor:adjoint is res} implies that 
\[
C^{*}\subseteq\left(\begin{array}{ccc}
-\dive & 0 & 0\end{array}\right),
\]
since $-\dive=\interior\grad^{\ast}$. In consequence (also use Lemma
\ref{lem:What is gradgg}), $(v,\eta_{1},\eta_{2})\in D(C^{*})$ if
and only if $v\in D(\dive)$ and $-\dive v=\grad^{\diamond}v+\gamma^{\diamond}\eta_{1}+\gamma'\grad_{\Gamma}^{\diamond}\eta_{2}\in L^{2}(\Omega)$,
which in turn is equivalent to $v\in D(\dive)$ and $n\cdot v+\eta_{1}+\grad_{\Gamma}^{\diamond}\eta_{2}=0$
on $L^{2}(\Gamma)$ by Lemma \ref{lem:boundary equation}. 
\end{proof}
In the rest of this section, we shall formally compute the equation
modeled by \eqref{eq:acoustic-bdy}. That is to say, we assume that
the system
\[
\left(\partial_{0}M_{0}+M_{1}+A\right)\left(\begin{array}{c}
p\\
\left(\begin{array}{c}
v\\
\eta_{1}\\
\eta_{2}
\end{array}\right)
\end{array}\right)=\left(\begin{array}{c}
f\\
\left(\begin{array}{c}
g\\
h_{1}\\
h_{2}
\end{array}\right)
\end{array}\right)
\]
given in \eqref{eq:acoustic-bdy} with the special block-diagonal
situation
\begin{align*}
M_{0}+\partial_{0}^{-1}M_{1} & =\left(\begin{array}{cc}
m_{\text{0}0} & \left(\begin{array}{ccc}
0 & 0 & 0\end{array}\right)\\
\left(\begin{array}{c}
0\\
0\\
0
\end{array}\right) & \left(\begin{array}{ccc}
\mu_{11} & 0 & 0\\
0 & \mu_{22} & 0\\
0 & 0 & \mu_{33}
\end{array}\right)
\end{array}\right)+\partial_{0}^{-1}\left(\begin{array}{cc}
n_{\text{0}0} & \left(\begin{array}{ccc}
0 & 0 & 0\end{array}\right)\\
\left(\begin{array}{c}
0\\
0\\
0
\end{array}\right) & \left(\begin{array}{ccc}
\nu_{11} & 0 & 0\\
0 & \nu_{22} & 0\\
0 & 0 & \nu_{33}
\end{array}\right)
\end{array}\right),
\end{align*}
for non-negative scalars $m_{00},n_{00},\mu_{11},\mu_{22},\mu_{33},\nu_{11},\nu_{22},\nu_{33}$
being arranged in the way that $M_{0}$ and $M_{1}$ satisfy the conditions
to warrant well-posedness, i.e., $M_{0}=M_{0}^{*}$ strictly positive
definite on its range and $\Re M_{1}$ strictly positive definite
on the nullspace of $M_{0}$, admits smooth solutions $\left(\begin{array}{c}
p\\
\left(\begin{array}{c}
v\\
\eta_{1}\\
\eta_{2}
\end{array}\right)
\end{array}\right)$. The resulting system leads to 
\begin{align}
\partial_{0}m_{00}p+n_{00}p+\dive v & =f\nonumber \\
\partial_{0}\mu_{11}v+\nu_{11}v+\grad p & =g\nonumber \\
\partial_{0}\mu_{22}\eta_{1}+\nu_{22}\eta_{1}+\gamma p & =h_{1}\nonumber \\
\partial_{0}\mu_{33}\eta_{2}+\nu_{33}\eta_{2}+\grad_{\Gamma}\gamma p & =h_{2}\nonumber \\
n\cdot v+\eta_{1} & =-\grad_{\Gamma}^{\diamond}\eta_{2},\label{eq:wave}
\end{align}
where the last equality is the boundary condition induced by the domain
of $A$. Eliminating $\eta_{1},\eta_{2}$ by using the third and fourth
line of \eqref{eq:wave}, we obtain by formal calculations for the
boundary condition
\begin{align*}
n\cdot v+\left(\partial_{0}\mu_{22}+\nu_{22}\right)^{-1}\left(h_{1}-\gamma p\right) & =-\grad_{\Gamma}^{\diamond}\left(\partial_{0}\mu_{33}+\nu_{33}\right)^{-1}\left(h_{2}-\grad_{\Gamma}\gamma p\right).
\end{align*}
Multiplication by $\left(\partial_{0}\mu_{33}+\nu_{33}\right)$ gives
\begin{multline*}
\left(\partial_{0}\mu_{33}+\nu_{33}\right)n\cdot v-\left(\partial_{0}\mu_{33}+\nu_{33}\right)\left(\partial_{0}\mu_{22}+\nu_{22}\right)^{-1}\gamma p-\grad_{\Gamma}^{\diamond}\,\grad_{\Gamma}\gamma p\\
=-\grad_{\Gamma}^{\diamond}h_{2}-\left(\partial_{0}\mu_{33}+\nu_{33}\right)\left(\partial_{0}\mu_{22}+\nu_{22}\right)^{-1}h_{1}.
\end{multline*}
With $\mu_{22}=0$ and $\nu_{22}=1$ this simplifies further to 
\[
\left(\partial_{0}\mu_{33}+\nu_{33}\right)n\cdot v-\left(\partial_{0}\mu_{33}+\nu_{33}\right)\gamma p-\grad_{\Gamma}^{\diamond}\grad_{\Gamma}\gamma p=-\grad_{\Gamma}^{\diamond}h_{2}-\left(\partial_{0}\mu_{33}+\nu_{33}\right)h_{1}.
\]
Eliminating $v$, by using the second equation of \eqref{eq:wave},
yields
\begin{multline*}
-\left(\partial_{0}\mu_{33}+\nu_{33}\right)n\cdot\left(\partial_{0}\mu_{11}+\nu_{11}\right)^{-1}\left(g-\grad p\right)+\left(\partial_{0}\mu_{33}+\nu_{33}\right)\gamma p+\grad_{\Gamma}^{\diamond}\,\grad_{\Gamma}\gamma p\\
=\grad_{\Gamma}^{\diamond}h_{2}+\left(\partial_{0}\mu_{33}+\nu_{33}\right)h_{1}
\end{multline*}
and hence,
\begin{multline*}
\left(\partial_{0}\mu_{33}+\nu_{33}\right)\left(\partial_{0}\mu_{11}+\nu_{11}\right)^{-1}n\cdot\grad p+\left(\partial_{0}\mu_{33}+\nu_{33}\right)\gamma p+\grad_{\Gamma}^{\diamond}\,\grad_{\Gamma}\gamma p\\
=\grad_{\Gamma}^{\diamond}h_{2}+\left(\partial_{0}\mu_{33}+\nu_{33}\right)h_{1}+\left(\partial_{0}\mu_{33}+\nu_{33}\right)\left(\partial_{0}\mu_{11}+\nu_{11}\right)^{-1}n\cdot g.
\end{multline*}
Simplifying further by letting $\mu_{33}=\mu_{11}\alpha$, $\nu_{33}=\nu_{11}\alpha$
with some non-zero $\alpha\in\mathbb{R}_{>0}$, we arrive at 
\begin{multline*}
\left(\partial_{0}\mu_{11}+\nu_{11}\right)\gamma p+n\cdot\grad p+\alpha^{-1}\grad_{\Gamma}^{\diamond}\grad_{\Gamma}\gamma p=\left(\partial_{0}\mu_{11}+\nu_{11}\right)h_{1}+n\cdot g+\alpha^{-1}\grad_{\Gamma}^{\diamond}h_{2},
\end{multline*}
which is a heat type equation on the boundary of $\Omega$ (see also
Remark \ref{rem:grad-dach, div}). Noting that the first and second
line of \eqref{eq:wave} gives 
\[
\partial_{0}m_{00}p+n_{00}p+\dive(\partial_{0}\mu_{00}+\nu_{00})^{-1}(g-\grad p)=f,
\]
which yields for $\nu_{00}=0$ an abstract wave equation of the form
\[
\partial_{0}^{2}m_{00}p+\partial_{0}n_{00}p-\dive\mu_{00}\grad p=\partial_{0}f-\dive\mu_{00}g.
\]
Thus, for suitable $M_{0}$ and $M_{1}$, \eqref{eq:acoustic-bdy}
models a wave equation on $\Omega$ with a heat equation on $\partial\Omega$
as boundary condition. 
\begin{rem}
In \cite{2013arXiv1312.5882D}, an equation was studied, where on
a $1$-codimensional submanifold in $\Omega$ another partial differential
equation on an interface is also considered. We shall outline here,
how to include such problems in the abstract setting just discussed.
So let $\Omega\subseteq\mathbb{R}^{d}$ open with Lipschitz boundary
and assume that there exists $\tilde{\Omega}\subseteq\Omega$ open
with the property that $\partial\Omega\subseteq\partial\tilde{\Omega}$
and such that both $\partial\tilde{\Omega}$ and $\partial\tilde{\Omega}\setminus\partial\Omega$
are locally the graph of a Lipschitz continuous function and are manifolds
allowing for the definition of a covariant derivative. As the operator
$C$, generating the abstract $\grad$--$\dive$ system, we take the
following operator
\[
C\coloneqq\left(\begin{array}{c}
\grad\\
\gamma\\
\grad_{\Gamma}\gamma
\end{array}\right)
\]
similar to the considerations above. The domain, however, changes
a little, namely $\gamma$ being now the evaluation at $\Gamma\coloneqq\partial\tilde{\Omega}$
and $\grad_{\Gamma}$ is given analogous to the above definition. 
\end{rem}

\subsection{A Leontovich type boundary condition.}

As for the acoustic equations just discussed, we formulate a similar
problem also for Maxwell's equations. We will show that Maxwell's
equations with impedance type boundary conditions can be formulated
within the framework of evolutionary equations employing the notion
of abstract $\grad$-$\dive$-systems. For this let $\Omega\subseteq\mathbb{R}^{3}$
open. In this section, we consider the equation 
\begin{equation}
\left(\partial_{0}M\left(\partial_{0}^{-1}\right)+A\right)\left(\begin{array}{c}
H\\
\left(\begin{array}{c}
E\\
\eta
\end{array}\right)
\end{array}\right)=\left(\begin{array}{c}
f\\
\left(\begin{array}{c}
g\\
h_{1}
\end{array}\right)
\end{array}\right)\in H_{\rho,0}\left(\mathbb{R};L^{2}(\Omega)^{3}\oplus L^{2}(\Omega)^{3}\oplus H_{\trace}\right).\label{eq:Max_reg}
\end{equation}
Here $H_{\trace}$ denotes a suitable Hilbert space, which will be
used to formulate the boundary condition. In the forthcoming subsections
we will discuss two possibilities for the choice of $H_{\trace}$.

\subsubsection{A classical trace version}

Throughout, assume that $\Omega$ is bounded and has Lipschitz-boundary
$\Gamma$ and denote by $n$ the unit outward normal vector-field
on $\Gamma$. In order to define $A$ in \eqref{eq:Max_reg} properly,
we need the following trace operators (see \cite{zbMATH01866780}).
\begin{defn}
Let $L_{\tau}^{2}(\Gamma)$ denote the space of tangential vector-fields
on $\Gamma,$ i.e. 
\[
L_{\tau}^{2}(\Gamma)\coloneqq\left\{ f\in L^{2}(\Gamma)^{3}\,|\, f\cdot n=0\right\} ,
\]
which is a closed subspace of $L^{2}(\Gamma)^{3}.$ We define the
mappings 
\[
\pi_{\tau}:H_{1}(\Omega)^{3}\to L_{\tau}^{2}(\Gamma),H\mapsto-n\times(n\times\gamma H)=\gamma H-(n\cdot\gamma H)n,
\]
and 
\[
\gamma_{\tau}:H_{1}(\Omega)^{3}\to L_{\tau}^{2}(\Gamma),H\mapsto\gamma H\times n,
\]
where $\gamma:H_{1}(\Omega)^{3}\to H_{1/2}(\Omega)^{3}$ denotes the
classical trace (cp. Subsection \ref{sub:wave}). We set $V_{\pi}\coloneqq\pi_{\tau}\left[H_{1}(\Omega)^{3}\right]$
and $V_{\gamma}\coloneqq\gamma_{\tau}\left[H_{1}(\Omega)^{3}\right]$,
which are Hilbert spaces equipped with the norms 
\begin{align*}
|v|_{V_{\pi}} & \coloneqq\inf_{H\in H_{1}(\Omega)^{3}}\{|\gamma H|_{H_{1/2}(\Gamma)^{3}}\,|\,\pi_{\tau}H=v\},\\
|v|_{V_{\gamma}} & \coloneqq\inf_{H\in H_{1}(\Omega)^{3}}\{|\gamma H|_{H_{1/2}(\Gamma)^{3}}\,|\,\gamma_{\tau}H=v\}.
\end{align*}
\end{defn}
\begin{rem}
$\,$

\begin{enumerate}[(a)]

\item We have $\id_{\pi}\colon V_{\pi}\hookrightarrow L_{\tau}^{2}(\Gamma)$
and $\id_{\gamma}\colon V_{\gamma}\hookrightarrow L_{\tau}^{2}(\Gamma)$
with continuous and dense embeddings. Consequently, $\id_{\pi}^{\diamond}\colon L_{\tau}^{2}(\Gamma)\hookrightarrow V_{\pi}'$
as well as $\id_{\gamma}^{\diamond}\colon L_{\tau}^{2}(\Gamma)\hookrightarrow V_{\gamma}'$
with continuous and dense embeddings.

\item Integration by parts gives 
\[
\langle\pi_{\tau}H|\gamma_{\tau}E\rangle_{L_{\tau}^{2}(\Gamma)}=\langle\curl H|E\rangle_{L^{2}(\Omega)^{3}}-\langle H|\curl E\rangle_{L^{2}(\Omega)^{3}},
\]
for all $H,E\in H_{1}(\Omega)^{3},$ which yields that (see \cite{zbMATH01866780}
or the concise presentation in \cite[Section 4]{WeissStaffans2013})
\[
R_{\left(L_{\tau}^{2}\right)'}\circ\pi_{\tau}:H^{1}(\Omega)^{3}\subseteq H(\curl,\Omega)\to V_{\gamma}',
\]
where $R_{\left(L_{\tau}^{2}\right)'}$ is defined as in Section \ref{sec:Construction-of-Abstract},
is continuous, where $H(\curl,\Omega)$ denotes the domain of $\curl$
equipped with its graph norm. Hence, $\pi_{\tau}$ has a unique continuous
extension to an operator 
\[
\pi_{\tau}:H(\curl,\Omega)\to V_{\gamma}'.
\]
The same argument works for $\gamma_{\tau}$, we, thus, get that 
\[
\gamma_{\tau}:H(\curl,\Omega)\to V_{\pi}'
\]
is continuous. 

\end{enumerate}\end{rem}
\begin{lem}
\label{lem:closability of tilde C}The operator
\[
C\colon\pi_{\tau}^{-1}\left[\id_{\tau}^{\diamond}\left[L_{\tau}^{2}\left(\Gamma\right)\right]\right]\subseteq L^{2}(\Omega)^{3}\to L^{2}(\Omega)^{3}\oplus L_{\tau}^{2}(\Gamma),H\mapsto\left(\begin{array}{c}
-\curl\\
\pi_{\tau}
\end{array}\right)H
\]
is closed.\end{lem}
\begin{proof}
First of all, note that $\pi_{\tau}^{-1}\left[\id_{\tau}^{\diamond}\left[L_{\tau}^{2}\left(\Gamma\right)\right]\right]\subseteq D(\curl)$.
Next, let $(\phi_{n})_{n}$ in $D(C)$ with $\phi_{n}\to\phi$ in
$L^{2}\left(\Omega\right)^{3}$ as $n\to\infty$ and $\left(\pi_{\tau}\phi_{n}\right)_{n}$
convergent to some $\psi\in L_{\tau}^{2}(\Gamma)$ as well as $\left(\curl\phi_{n}\right)_{n}$
convergent to some $\eta\in L^{2}\left(\Omega\right)^{3}$. By the
closedness of $\curl$, we infer that $\phi\in H(\curl,\Omega)$ and
$\curl\phi=\eta$. Hence, $\phi_{n}\to\phi$ in $H(\curl,\Omega)$
as $n\to\infty$. As $\pi_{\tau}$ is continuous, we infer that $\pi_{\tau}\phi_{n}\to\pi_{\tau}\phi$
in $V_{\gamma}'$ as $n\to\infty$. Since $L_{\tau}^{2}(\Gamma)\hookrightarrow V_{\gamma}'$
with continuous embedding, we get $\pi_{\tau}\phi=\psi\in L_{\tau}^{2}(\Gamma)$,
which implies the assertion. 
\end{proof}
Define for this section
\begin{align*}
X_{0} & \coloneqq L^{2}(\Omega)^{3},\\
X_{1} & \coloneqq\left(D\left(C\right),\sqrt{|\cdot|^{2}+|C\cdot|^{2}}\right).
\end{align*}

We may now introduce the operator $A$: 
\[
A=\left(\begin{array}{cc}
\left(\begin{array}{ccc}
0 & 0 & 0\\
0 & 0 & 0\\
0 & 0 & 0
\end{array}\right) & -\left(\begin{array}{c}
-\widetilde{\curl}\\
\tilde{\pi_{\tau}}
\end{array}\right)^{*}\\
\left(\begin{array}{c}
-\widetilde{\curl}\\
\tilde{\pi_{\tau}}
\end{array}\right) & \left(\begin{array}{cc}
\left(\begin{array}{ccc}
0 & 0 & 0\\
0 & 0 & 0\\
0 & 0 & 0
\end{array}\right) & \left(\begin{array}{ccc}
0 & 0 & 0\\
0 & 0 & 0\\
0 & 0 & 0
\end{array}\right)\\
\left(\begin{array}{ccc}
0 & 0 & 0\\
0 & 0 & 0\\
0 & 0 & 0
\end{array}\right) & \left(\begin{array}{ccc}
0 & 0 & 0\\
0 & 0 & 0\\
0 & 0 & 0
\end{array}\right)
\end{array}\right)
\end{array}\right),
\]
where $\widetilde{\curl}\coloneqq\curl\restricted{X_{1}}:X_{1}\to L^{2}\left(\Omega\right)^{3}$
and $\tilde{\pi}_{\tau}\coloneqq\pi_{\tau}|_{X_{1}}:X_{1}\to L_{\tau}^{2}\left(\Gamma\right)$.
$\left(\begin{array}{c}
-\widetilde{\curl}\\
\tilde{\pi}_{\tau}
\end{array}\right):X_{1}\subseteq L^{2}(\Omega)^{3}\to L^{2}(\Omega)^{3}\oplus L_{\tau}^{2}(\Gamma)$ generates an abstract $\grad$--$\dive$ system: the only thing left
to show is that $X_{1}$ is dense in $L^{2}(\Omega)^{3}$. This, however,
is trivial as $\interior C_{\infty}(\Omega)\subseteq X_{1}$. 

Similarly to the previous section, we shall have a look at the condition
for being contained in the domain of $\left(\begin{array}{c}
-\widetilde{\curl}\\
\tilde{\pi}_{\tau}
\end{array}\right)^{*}$. Again, we need a prerequisite:
\begin{lem}
\label{lem:bdy eq Maxwell1}Let $E\in D(\curl)$, $\eta\in L_{\tau}^{2}\left(\Gamma\right)$.
Then $\curl E=\tilde{\curl}^{\diamond}E-\tilde{\pi}_{\tau}^{\diamond}\eta$
if and only if $\gamma_{\tau}E+\eta=0$ on $L_{\tau}^{2}\left(\Gamma\right)$
.\end{lem}
\begin{proof}
We observe that for $\Psi\in H^{1}(\Omega)^{3}\subseteq X_{1}$ the
equation
\begin{align}
\left(\gamma_{\tau}E+\eta\right)(\pi_{\tau}\Psi) & =\left\langle \curl E|\Psi\right\rangle _{L^{2}\left(\Omega\right)^{3}}-\left\langle E|\tilde{\curl}\Psi\right\rangle _{L^{2}\left(\Omega\right)^{3}}+\langle\eta|\tilde{\pi}_{\tau}\Psi\rangle_{L_{\tau}^{2}(\Gamma)}\label{eq:bdy_cdt_for_curl}\\
 & =\left\langle \curl E|\Psi\right\rangle _{L^{2}\left(\Omega\right)^{3}}-\left(\tilde{\curl}^{\diamond}E-\tilde{\pi}_{\tau}^{\diamond}\eta\right)(\Psi)\nonumber 
\end{align}
holds true. Thus, if $\curl E=\widetilde{\curl}^{\diamond}E-\tilde{\pi}_{\tau}^{\diamond}\eta,$
we get that $\left(\gamma_{\tau}E+\eta\right)(\pi_{\tau}\Psi)=0$
for each $\Psi\in H^{1}(\Omega)^{3}.$ Thus, $\gamma_{\tau}E+\eta=0$
on $L_{\tau}^{2}(\Gamma),$ due to the density of $\tilde{\pi}_{\tau}[H^{1}(\Omega)^{3}]=V_{\pi}$
in $L_{\tau}^{2}(\Gamma),$ see \cite[p 850]{zbMATH01866780}. On
the other hand, if $\gamma_{\tau}E+\eta=0,$ equation \eqref{eq:bdy_cdt_for_curl}
immediately gives $\curl E=\tilde{\curl}^{\diamond}E-\tilde{\pi}_{\tau}^{\diamond}\eta$
by the density of $H^{1}(\Omega)^{3}$ in $L^{2}(\Omega)^{3}.$ \end{proof}
\begin{thm}
\label{thm:bdy eq adjoint}We have $\left(\begin{array}{c}
-\widetilde{\curl}\\
\tilde{\pi_{\tau}}
\end{array}\right)^{*}\subseteq\left(\begin{array}{cc}
-\curl & 0\end{array}\right)$ and 
\[
D\left(\left(\begin{array}{c}
-\widetilde{\curl}\\
\tilde{\pi_{\tau}}
\end{array}\right)^{*}\right)=\left\{ (E,\eta)\in D(\curl)\times L_{\tau}^{2}\left(\Gamma\right)\,|\,\gamma_{\tau}E+\eta=0\text{ on }L_{\tau}^{2}(\Gamma)\right\} .
\]
\end{thm}
\begin{proof}
Note that with $\interior\curl=\curl^{*}$ we have 
\[
\left(\begin{array}{c}
-\interior\curl\\
0
\end{array}\right)\subseteq\left(\begin{array}{c}
-\widetilde{\curl}\\
\tilde{\pi_{\tau}}
\end{array}\right)\eqqcolon C.
\]
Hence, by Corollary \ref{cor:adjoint is res}, we get 
\[
C^{*}\subseteq\left(\begin{array}{cc}
-\curl & 0\end{array}\right).
\]
Therefore, by Theorem \ref{thm:AGDS-adjoint}, we obtain $(E,\eta)\in D(C^{*})$
if and only if $E\in D(\curl)$ and 
\[
-\curl E=-\widetilde{\curl}^{\diamond}E+\tilde{\pi_{\tau}}^{\diamond}\eta\in L^{2}(\Omega)^{3},
\]
which, in turn, by Lemma \ref{lem:bdy eq Maxwell1} is equivalent
to $\gamma_{\tau}E+\eta=0$ on $L_{\tau}^{2}\left(\Gamma\right)$
and $E\in D(\curl)$.
\end{proof}
The latter theorem tells us that the containment in the domain of
$A$ is a boundary equation. So, in order to have a general class
of Maxwell-type equations at hand, which include differential equations
on the boundary, any suitable material law operator $M(\partial_{0}^{-1})$
allowing for well-posedness of \eqref{eq:Max_reg} is admitted. 

We shall elaborate the following particular choice for the material
law operator: 
\begin{align*}
M\left(\partial_{0}^{-1}\right)= & \left(\begin{array}{cc}
\mu & \left(\begin{array}{cc}
0 & 0\end{array}\right)\\
\left(\begin{array}{c}
0\\
0
\end{array}\right) & \left(\begin{array}{cc}
\epsilon & 0\\
0 & \kappa\left(\partial_{0}^{-1}\right)
\end{array}\right)
\end{array}\right)\in L\left(H_{\rho,0}\left(\mathbb{R};L^{2}(\Omega)^{3}\oplus L^{2}\left(\Omega\right)^{3}\oplus L_{\tau}^{2}\left(\Gamma\right)\right)\right),
\end{align*}
where $\kappa:B_{\mathbb{C}}\left(\frac{1}{2\rho_{0}},\frac{1}{2\rho_{0}}\right)\to L(L_{\tau}^{2}(\Gamma))$
is analytic and satisfies condition \eqref{eq:pos-def}. For this
choice (also use the result from Theorem \ref{thm:bdy eq adjoint})
\eqref{eq:Max_reg} formally reads as
\begin{align*}
\partial_{0}\mu H+\curl E & =f,\\
\partial_{0}\epsilon E-\widetilde{\curl}H & =g,\\
\partial_{0}\kappa\left(\partial_{0}^{-1}\right)\eta+\tilde{\pi_{\tau}}H & =h_{1},\\
\gamma_{\tau}E+\eta & =0.
\end{align*}
Elimination of $\eta$ yields
\begin{equation}
-\partial_{0}\kappa\left(\partial_{0}^{-1}\right)\gamma_{\tau}E+\tilde{\pi_{\tau}}H=h_{1}.\label{eq:bdy eq imp}
\end{equation}
This is an impedance type boundary condition. To see this we recall
$\gamma_{\tau}E=\gamma E\times n$ and $\pi_{\tau}H=-n\times(n\times\gamma H)$,
use that $\left(n\times\right)\left(-n\times\right)\left(n\times\right)=\left(n\times\right)$
and multiply \eqref{eq:bdy eq imp} by $-n\times$. For ease of formulation,
we shall suppress denoting the trace operator $\gamma$. We obtain
\[
n\times\left(\partial_{0}\kappa\left(\partial_{0}^{-1}\right)n\times E\right)-n\times H=-n\times h_{1},
\]
or, equivalently,
\[
-n\times\left(\partial_{0}\kappa(\partial_{0}^{-1})n\times\left(n\times\left(n\times E\right)\right)\right)=n\times H-n\times h_{1},
\]
which gives,
\begin{equation}
\pi_{\tau}E=\left(\left(-n\times\right)\partial_{0}\kappa\left(\partial_{0}^{-1}\right)\left(n\times\right)\right)^{-1}\left(n\times h_{1}-H\times n\right).\label{eq:Leontovich}
\end{equation}
Now, in the literature, we find several choices for the impedance
operator 
\[
Z\left(\partial_{0}^{-1}\right)\coloneqq\left(\left(-n\times\right)\partial_{0}\kappa\left(\partial_{0}^{-1}\right)\left(n\times\right)\right)^{-1}.
\]
For discussing some of these choices, we let $\epsilon^{\prime},\mu^{\prime},\sigma^{\prime},\tau^{\prime}\in\mathbb{R}_{\geq0}$
be parameters with $\epsilon'+\sigma'>0$ and $\mu'+\tau'>0$. Mohsen
\cite{Mohsen1982} used the following form
\[
Z\left(\partial_{0}^{-1}\right)=\left(\mu^{\prime}+\tau^{\prime}\partial_{0}^{-1}\right)^{\frac{1}{2}}\left(\epsilon^{\prime}+\sigma^{\prime}\partial_{0}^{-1}\right)^{-\frac{1}{2}},
\]
leading to a boundary condition, which Mohsen attributed to Rytov
and Leontovich. Senior, \cite{doi:10.1139/p62-067}, and De Santis
e.a., \cite{de2012efficient}, discussed the case $\tau'=0$, resulting
in 
\[
Z\left(\partial_{0}^{-1}\right)=\sqrt{\frac{\mu'}{\epsilon'}}\left(1+\frac{\sigma'}{\epsilon'}\partial_{0}^{-1}\right)^{-\frac{1}{2}}.
\]
In the eddy current approximation ($\epsilon'=0$) and also assuming
$\tau'=0$, a case discussed by Hiptmair, L\'opez-Fern\'andez \&
Paganini, \cite{HLP13_1073}, we arrive at the following form
\[
Z\left(\partial_{0}^{-1}\right)=\sqrt{\frac{\mu'}{\sigma'}}\partial_{0}^{-\frac{1}{2}}.
\]
This leads to fractional evolution on the boundary. 
\begin{rem}
In \cite{zbMATH01975976} the author considers Maxwell's equations
in a two-dimensional setting and imposes a family of boundary conditions
on a surface $\Sigma$ contained in an exterior domain, which can
be written in the form
\begin{equation}
\partial_{0}\left(\gamma_{\tau}E-\pi_{\tau}H\right)=\left(\frac{k}{4}-\gamma\right)\gamma_{\tau}E+\left(\frac{k}{4}+\gamma\right)\pi_{\tau}H,\label{eq:Burque}
\end{equation}
where $k$ denotes the curvature of $\Sigma$ and $\gamma$ is an
arbitrary function on $\Sigma.$ It is shown that each of those boundary
conditions yield a well-posed system and the asymptotic behavior of
the corresponding solutions are studied. We remark here that the latter
boundary condition is covered by (\ref{eq:bdy eq imp}). Indeed, choosing
\[
\kappa\left(\partial_{0}^{-1}\right)=\partial_{0}^{-1}-\partial_{0}^{-1}\frac{k}{2}\left(\partial_{0}+\frac{k}{4}+\gamma\right)^{-1}
\]
and $h_{1}=0$ we obtain 
\begin{align*}
0 & =-\partial_{0}\kappa\left(\partial_{0}^{-1}\right)\gamma_{\tau}E+\pi_{\tau}H\\
 & =-\left(1-\frac{k}{2}\left(\partial_{0}+\frac{k}{4}+\gamma\right)^{-1}\right)\gamma_{\tau}E+\pi_{\tau}H,
\end{align*}
which gives 
\[
0=-\left(\partial_{0}-\frac{k}{4}+\gamma\right)\gamma_{\tau}E+\left(\partial_{0}+\frac{k}{4}+\gamma\right)\pi_{\tau}H,
\]
which is just a reformulation of (\ref{eq:Burque}).
\end{rem}

\subsubsection{An abstract trace version}

In this last section, we shall provide a possible way of by-passing
boundary regularity requirements giving the corresponding Leontovich
type boundary conditions its most general form. For doing so, we need
the following definitions, which are adapted from the abstract boundary
data spaces, see e.g.~\cite{Trostorff2014,zbMATH06295127,Waurick2014IMAJMCI_ComprehContr}
\begin{defn}
Recall that $\interior\curl\subseteq\curl$, and consequently $H(\interior\curl,\Omega)$
is a closed subspace of $H(\curl,\Omega).$ We define the \emph{boundary
data space of $\curl$ }by 
\[
\mathcal{BD}(\curl)\coloneqq H(\interior\curl,\Omega)^{\bot},
\]
where the orthogonal complement is taken in $H(\curl,\Omega).$ Moreover,
we denote by $\iota_{\curl}:\mathcal{BD}(\curl)\to H(\curl,\Omega)$
the canonical embedding, i.e. $\iota_{\curl}\Phi=\Phi$. Then, $\iota_{\curl}^{\ast}\iota_{\curl}=\mathrm{id}_{\mathcal{BD}(\curl)}$
and 
\begin{align*}
\iota_{\curl}\iota_{\curl}^{\ast}:H(\curl,\Omega) & \to H(\curl,\Omega)
\end{align*}
is the orthogonal projector on $\mathcal{BD}(\curl).$ An easy computation
shows that 
\begin{align*}
\mathcal{BD}(\curl) & =N(1+\curl\curl),
\end{align*}
which in particular yields that $\curl[\mathcal{BD}(\curl)]\subseteq\mathcal{BD}(\curl).$
We set 
\begin{align*}
\stackrel{\bullet}{\curl}:\mathcal{BD}(\curl) & \to\mathcal{BD}(\curl)\\
\phi & \mapsto\iota_{\curl}^{\ast}\curl\iota_{\curl}
\end{align*}
Consequently, we have $\stackrel{\bullet}{\curl}\stackrel{\bullet}{\curl}=-\mathrm{id}_{\mathcal{BD}(\curl)}.$
\end{defn}
The main idea is now to replace the space $L_{\tau}^{2}(\Gamma)$
in the previous section by the space $\mathcal{BD}(\curl)$ and the
trace operator $\pi_{\tau}$ by the operator $\iota_{\curl}^{\ast}.$
Indeed, setting $X_{0}\coloneqq L^{2}(\Omega)^{3}$ and $X_{1}\coloneqq H(\curl,\Omega)$
we obtain that $\left(\begin{array}{c}
-\curl\\
\iota_{\curl}^{\ast}
\end{array}\right):H(\curl,\Omega)\subseteq L^{2}(\Omega)^{3}\to L^{2}(\Omega)^{3}\oplus\mathcal{BD}(\curl)$ generates an abstract $\grad$--$\dive$ system, see also Proposition
\ref{prop:standard_grad-div}. Consequently, \eqref{eq:Max_reg} with
\[
A\coloneqq\left(\begin{array}{cc}
\left(\begin{array}{ccc}
0 & 0 & 0\\
0 & 0 & 0\\
0 & 0 & 0
\end{array}\right) & -\left(\begin{array}{c}
-\curl\\
\iota_{\curl}^{\ast}
\end{array}\right)^{\ast}\\
\left(\begin{array}{c}
-\curl\\
\iota_{\curl}^{\ast}
\end{array}\right) & \left(\begin{array}{cc}
\left(\begin{array}{ccc}
0 & 0 & 0\\
0 & 0 & 0\\
0 & 0 & 0
\end{array}\right) & \left(\begin{array}{ccc}
0 & 0 & 0\\
0 & 0 & 0\\
0 & 0 & 0
\end{array}\right)\\
\left(\begin{array}{ccc}
0 & 0 & 0\\
0 & 0 & 0\\
0 & 0 & 0
\end{array}\right) & \left(\begin{array}{ccc}
0 & 0 & 0\\
0 & 0 & 0\\
0 & 0 & 0
\end{array}\right)
\end{array}\right)
\end{array}\right)
\]
is well-posed in $H_{\rho,0}\left(\mathbb{R};L^{2}(\Omega)^{3}\oplus L^{2}(\Omega)^{3}\oplus\mathcal{BD}(\curl)\right)$.
In order to characterize the domain of $\left(\begin{array}{c}
-\curl\\
\iota_{\curl}^{\ast}
\end{array}\right)^{\ast}$, we need to understand the operator $\left(\iota_{\curl}^{\ast}\right)^{\diamond}:\mathcal{BD}(\curl)\to H(\curl,\Omega)'$.
\begin{lem}
\label{lem:BD(curl)}We have $\left(\iota_{\curl}^{\ast}\right)^{\diamond}=(1+\curl^{\diamond}\curl)\iota_{\curl}.$\end{lem}
\begin{proof}
For $\Psi\in H(\curl,\Omega),\Phi\in\mathcal{BD}(\curl)$ we compute
\begin{align*}
\left(\iota_{\curl}^{\ast}\right)^{\diamond}\Phi(\Psi) & =\langle\Phi|\iota_{\curl}^{\ast}\Psi\rangle_{\mathcal{BD}(\curl)}\\
 & =\langle\iota_{\curl}\Phi|\Psi\rangle_{H(\curl,\Omega)}\\
 & =\langle\iota_{\curl}\Phi|\Psi\rangle_{L^{2}(\Omega)^{3}}+\langle\curl\iota_{\curl}\Phi|\curl\Psi\rangle_{L^{2}(\Omega)^{3}}\\
 & =\left(\left(1+\curl^{\diamond}\curl\right)\iota_{\curl}\Phi\right)(\Psi),
\end{align*}
which gives the assertion.\end{proof}
\begin{rem}
Note that in Lemma \ref{lem:BD(curl)}, $1=R_{\left(L^{2}(\Omega)^{3}\right)'}$,
the latter operator being given in Section \ref{sec:Construction-of-Abstract}.
\end{rem}
Thus, we end up with the following characterization result.
\begin{thm}
\label{thm:adjoint_BD}We have $\left(\begin{array}{c}
-\curl\\
\iota_{\curl}^{\ast}
\end{array}\right)^{*}\subseteq\left(\begin{array}{cc}
-\curl & 0\end{array}\right)$ and 
\[
D\left(\left(\begin{array}{c}
-\curl\\
\iota_{\curl}^{\ast}
\end{array}\right)^{*}\right)=\left\{ (E,\eta)\in D(\curl)\times\mathcal{BD}(\curl)\,\left|\,\iota_{\curl}^{\ast}E-\stackrel{\bullet}{\curl}\eta=0\right.\right\} .
\]
\end{thm}
\begin{proof}
Since $\left(\begin{array}{c}
-\interior\curl\\
0
\end{array}\right)\subseteq\left(\begin{array}{c}
-\curl\\
\iota_{\curl}^{\ast}
\end{array}\right),$ we obtain $\left(\begin{array}{c}
-\curl\\
\iota_{\curl}^{\ast}
\end{array}\right)^{*}\subseteq\left(\begin{array}{cc}
-\curl & 0\end{array}\right)$ by Corollary \ref{cor:adjoint is res}. Hence, by Theorem \ref{thm:AGDS-adjoint}
we have $(E,\eta)\in D\left(\left(\begin{array}{c}
-\curl\\
\iota_{\curl}^{\ast}
\end{array}\right)^{\ast}\right)$ if and only if $E\in D(\curl)$ and 
\begin{align*}
-\curl E & =-\curl^{\diamond}E+\left(\iota_{\curl}^{\ast}\right)^{\diamond}\eta\\
 & =-\curl^{\diamond}E+(1+\curl^{\diamond}\curl)\iota_{\curl}\eta,
\end{align*}
by Lemma \ref{lem:BD(curl)}. The latter gives 
\begin{align*}
\left(\curl^{\diamond}-\curl\right)E & =\left(1+\curl^{\diamond}\curl\right)\iota_{\curl}\eta\\
 & =(-\curl\curl+\curl^{\diamond}\curl)\iota_{\curl}\eta\\
 & =\left(\curl^{\diamond}-\curl\right)\curl\iota_{\curl}\eta,
\end{align*}
where we have used $\mathcal{BD}(\curl)=N(1+\curl\curl).$ The latter
is equivalent to 
\[
\iota_{\curl}\iota_{\curl}^{\ast}\left(E-\curl\iota_{\curl}\eta\right)=0.
\]
Indeed, we have that an element $\Psi\in D(\curl)$ satisfies $(\curl^{\diamond}-\curl)\Psi=0$
if and only if for each $\Phi\in D(\curl)$ we have 
\begin{align*}
0 & =\left(\left(\curl^{\diamond}-\curl\right)\Psi\right)(\Phi)\\
 & =\langle\Psi|\curl\Phi\rangle_{L^{2}(\Omega)^{3}}-\langle\curl\Psi|\Phi\rangle_{L^{2}(\Omega)},
\end{align*}
which in turn is equivalent to $\Psi\in D(\curl^{\ast})=D(\interior\curl).$
Applying this to $E-\curl\iota_{\curl}\eta\in D(\curl),$ we obtain
the equivalence of 
\[
\left(\curl^{\diamond}-\curl\right)E=\left(\curl^{\diamond}-\curl\right)\curl\iota_{\curl}\eta
\]
and
\begin{equation}
\iota_{\curl}\iota_{\curl}^{\ast}\left(E-\curl\iota_{\curl}\eta\right)=0.\label{eq:bdy_eqn_curl}
\end{equation}
Using the definition of $\stackrel{\bullet}{\curl}$ and the injectivity
of $\iota_{\curl}$, we deduce the equivalence of \eqref{eq:bdy_eqn_curl}
to
\[
\iota_{\curl}^{\ast}E-\stackrel{\bullet}{\curl}\eta=0.\tag*{{\qedhere}}
\]

\end{proof}
In the remaining part of this subsection, we like to point out that
the boundary condition formulated in Theorem \ref{thm:adjoint_BD}
is indeed a generalization of the classical boundary condition in
Theorem \ref{thm:bdy eq adjoint}. For doing so, let $(E,\eta)\in D(\curl)\times\mathcal{BD}(\curl)$
with $\iota_{\curl}^{\ast}E-\stackrel{\bullet}{\curl}\eta=0.$ Note
that the latter is equivalent to $\stackrel{\bullet}{\curl}\iota_{\curl}^{\ast}E+\eta=0,$
since $\stackrel{\bullet}{\curl}\stackrel{\bullet}{\curl}=-\mathrm{id}_{\mathcal{BD}(\curl)}$.
Thus, we have to show that $\stackrel{\bullet}{\curl}\iota_{\curl}^{\ast}E$
can be identified with $\gamma_{\tau}E$ in a certain sense. We compute
for $E,H\in C^{\infty}(\overline{\Omega})$
\begin{align*}
\langle\stackrel{\bullet}{\curl}\iota_{\curl}^{\ast}E|\iota_{\curl}^{\ast}H\rangle_{\mathcal{BD}(\curl)} & =\langle\iota_{\curl}\stackrel{\bullet}{\curl}\iota_{\curl}^{\ast}E|\iota_{\curl}\iota_{\curl}^{\ast}H\rangle_{H(\curl,\Omega)}\\
 & =\langle\curl\iota_{\curl}\stackrel{\bullet}{\curl}\iota_{\curl}^{\ast}E|\curl\iota_{\curl}\iota_{\curl}^{\ast}H\rangle_{L^{2}(\Omega)^{3}}+\\
 & \quad+\langle\iota_{\curl}\stackrel{\bullet}{\curl}\iota_{\curl}^{\ast}E|\iota_{\curl}\iota_{\curl}^{\ast}H\rangle_{L^{2}(\Omega)^{3}}\\
 & =-\langle\iota_{\curl}\iota_{\curl}^{\ast}E|\curl\iota_{\curl}\iota_{\curl}^{\ast}H\rangle_{L^{2}(\Omega)^{3}}+\langle\curl\iota_{\curl}\iota_{\curl}^{\ast}E|\iota_{\curl}\iota_{\curl}^{\ast}H\rangle_{L^{2}(\Omega)^{3}}\\
 & =-\langle E|\curl H\rangle_{L^{2}(\Omega)^{3}}+\langle\curl E|H\rangle_{L^{2}(\Omega)^{3}}\\
 & =\intop_{\Gamma}(n\times E)^{\ast}H
\end{align*}
in case of a smooth boundary $\Gamma\coloneqq\partial\Omega$. Thus,
indeed $n\times E$ can be identified with $\stackrel{\bullet}{\curl}\iota_{\curl}^{\ast}E$,
or, in other words, $\stackrel{\bullet}{\curl}:\mathcal{BD}(\curl)\to\mathcal{BD}(\curl)$
is a suitable generalization of the operator $n\times:L_{\tau}^{2}(\Gamma)\to L_{\tau}^{2}(\Gamma).$
Indeed, both operators are unitary and their adjoints are the negative
operators, i.e. $(n\times)^{\ast}=-n\times$ as well as $(\stackrel{\bullet}{\curl})^{\ast}=-\stackrel{\bullet}{\curl}.$

\end{document}